\documentclass[12pt, leqno,a4paper]{amsart}
\usepackage{amssymb,enumerate}
\usepackage[T1]{fontenc} 
\usepackage[utf8]{inputenc}
\usepackage{import}
\usepackage{longtable}
\usepackage{tabularx}

\usepackage[english]{babel}
\usepackage{mathtools}
\usepackage{amssymb, amsthm, leftindex}
\usepackage{MnSymbol}
\usepackage{graphicx}
\usepackage[colorinlistoftodos]{todonotes}
\setuptodonotes{noline, color = cyan!60}
\usepackage{placeins}
\usepackage{appendix}

\usepackage[backend=biber, sorting=nyt, giveninits=true, date=year, style=alphabetic, isbn=false, url=false, doi=false]{biblatex}
\renewbibmacro{in:}{}
\DeclareFieldFormat{pages}{#1}
\DeclareFieldFormat{postnote}{#1}
\DeclareFieldFormat{multipostnote}{#1}
\usepackage{csquotes}

\usepackage{tikz-cd}

\newcommand{\C}{\mathbb{C}}
\newcommand{\K}{\mathbf{K}}
\renewcommand{\k}{\mathbf{k}}

\newcommand{\Ql}{\overline{\mathbb{Q}}_\ell}
\newcommand{\N}{\mathbb{N}}
\newcommand{\Z}{\mathbb{Z}}
\newcommand{\F}{\mathbb{F}}

\newcommand{\Oo}{\mathcal{O}}

\newcommand{\SL}{\operatorname{SL}}

\newcommand{\Ind}{\operatorname{Ind}}
\newcommand{\Res}{\operatorname{Res}}
\newcommand{\Inf}{\operatorname{Inf}}
\newcommand{\Ad}{\operatorname{Ad}}
\newcommand{\ad}{\operatorname{ad}}

\newcommand{\Hom}{\operatorname{Hom}}
\newcommand{\End}{\operatorname{End}}

\newcommand{\G}{\mathbf{G}}
\renewcommand{\H}{\mathbf{H}}
\newcommand{\T }{\mathbf{T}}
\renewcommand{\L}{\mathbf{L}}
\newcommand{\B}{\mathbf{B}}
\newcommand{\U}{\mathbf{U}}
\renewcommand{\P}{\mathbf{P}}

\newcommand{\Ll}{\mathcal{L}}
\newcommand{\Ee}{\mathcal{E}}
\newcommand{\Aa}{\mathcal{A}}

\newcommand{\irr}{\operatorname{Irr}}

\newcommand{\g}{\mathfrak{g}}

\newcommand{\defeq}{\coloneqq}

\newcommand{\sgn}{\operatorname{sgn}}
\newcommand{\pr}{\operatorname{pr}}

\newcommand{\Mm}{\mathcal{M}}

\newcommand{\Cc}{\mathcal{C}}
\newcommand{\Ff}{\mathcal{F}}

\newcommand{\Uch}{\operatorname{Uch}}
\newcommand{\Lie}{\operatorname{Lie}}

\newtheorem{thm}{Theorem}[section]
\newtheorem*{thm*}{Theorem}

\newtheorem{cor}[thm]{Corollary}
\newtheorem*{cor*}{Corollary}

\newtheorem*{conj*}{Conjecture}
\newtheorem{lem}[thm]{Lemma}
\newtheorem{prop}[thm]{Proposition}
\newtheorem*{prop*}{Proposition}

\theoremstyle{definition}

\newtheorem{dfn}[thm]{Definition}

\theoremstyle{remark}

\newtheorem{rmk}[thm]{Remark}

\newtheorem*{hyp}{Hypothesis}
\newtheorem*{hyps}{Hypotheses}

\newtheoremstyle%
{blue2thm}%
{}{}%
{\color{blue!50!red}\itshape}
{}%
{\color{blue!25!red}\bfseries}%
{\color{blue!25!red}.}%
{ }{}
\theoremstyle{blue2thm}

\title[Shape of decomposition matrices for exceptional finite reductive groups]{Unitriangularity of decomposition matrices of the unipotent $\ell$-blocks for simple adjoint exceptional groups}
\author[M. Roth]{Marie Roth}
	\address{FB Mathematik, RPTU Kaiserslautern-Landau, Postfach 3049, 67653 Kaiserslautern, Germany}
\email{roth@mathematik.uni-kl.de}

\begin{document}
	\begin{abstract}
	In 2020, Brunat--Dudas--Taylor showed that the decomposition matrix of unipotent $\ell$-blocks of a finite reductive group in good characteristic has unitriangular shape, under some conditions on the prime $\ell$, in particular $\ell$ being good. We extend this result to $\ell$ bad by adapting their proof to include the $\ell$-special classes defined by Chaneb.
	\end{abstract}
	
	\maketitle
	
	
	\tableofcontents
	
\bigskip
One approach to the representation theory of a finite group $G$ is to link the ordinary representations of $G$ to its modular representations, defined over a field of positive characteristic $\ell >0$. In other words, it boils down to knowing the $\ell$-decomposition matrix of~$G$. In general, finding the decomposition numbers is a difficult problem, which is not solved yet even for symmetric groups. However, as a first approximation, one could focus on the shape of the $\ell$-decomposition matrix. In his PhD thesis \cite{geck1990verallgemeinerte}, Geck stated a conjecture for finite groups of Lie type: \begin{conj*}[Geck]
	If $G$ is a finite group of Lie type of characteristic different from $\ell$, then its $\ell$-decomposition matrix is unitriangular.
\end{conj*}
Recall that $\ell$-decomposition matrices have a block decomposition along the blocks of $G$. One union of blocks of particular interest is the unipotent $\ell$-blocks. 
In \cite[Thm.~A]{brunatUnitriangularShapeDecomposition2020}, Brunat--Dudas--Taylor gave a positive answer to Geck's conjecture for these blocks under some mild assumptions on the prime~$\ell$. 

\begin{thm*}[Brunat--Dudas--Taylor]\label{thm_TriangUnipGood}
	Let $\G$ be a connected reductive group over ${k}$, an algebraically closed field of characteristic $p > 0$. Let $F : \G \to \G$ be a Frobenius endomorphism endowing $\G$ with an $\F_q$-structure, for $q$ a power of the prime $p$. Assume the following:
	\begin{enumerate}
		\item $p \neq \ell$ and $p$ is good for $\G$,
		\item $\ell$ is good for $\G$ and $\ell$ does not divide the order of the largest quotient of $Z(\G)$ on which $F$ acts trivially.
	\end{enumerate}
	Then the $\ell$-decomposition matrix of the unipotent $\ell$-blocks of $\G^F$ is lower-unitriangular, with respect to a suitable ordering. 
\end{thm*}

We extend the result to $\ell$ bad and $\G$ of exceptional type.

\begin{thm*}\label{mainresult}
Let $\G$ be an adjoint simple algebraic group of exceptional type defined over ${k}$, an algebraically closed field of characteristic $p > 0$. Let $F : \G \to \G$ be a Frobenius endomorphism. Assume that $p$ is good for $\G$ and $\ell$ is bad for $\G$. Then the $\ell$-decomposition matrix of the unipotent $\ell$-blocks of $\G^F$ is lower-unitriangular. 
\end{thm*}

As a corollary to Brunat--Dudas--Taylor's theorem and thanks to a result of Geck for classical groups when $\ell = 2$ \cite[Thm.~2.5]{geckBasicSetsBrauer1994}, we conclude the following.  

\begin{cor*}
Let $\G$ be an adjoint simple algebraic group defined over ${k}$, an algebraically closed field of characteristic $p > 0$. Let $F : \G \to \G$ be a Frobenius endomorphism. Assume that $p$ is good for $\G$. If  $\ell \neq p$, then the $\ell$-decomposition matrix of the unipotent $\ell$-blocks of $\G^F$ is lower-unitriangular. 
\end{cor*}

\subsubsection*{On the methods} In extending the techniques developed by Brunat--Dudas--Taylor, we encounter two main obstructions. Firstly, the unipotent characters form a basic set for the unipotent block only when $\ell$ is good. The ordering of the basic set is then based on the unipotent support of these unipotent characters, i.e. on the special unipotent conjugacy classes of $\G$. In our case, we have to choose another basic set, based on some characters of $\G$ with $\ell$-special unipotent support (as described by Chaneb \cite{chanebBasicSetsDecomposition2019}).\\
The second difference lies in the choice of projective characters. In their paper, Brunat--Dudas--Taylor defined and used Kawanaka characters, constructed from ordinary characters of certain small groups. When $\ell$ is bad, this construction might not be projective. We therefore extend the definition of Kawanaka characters as coming from projective characters of the same small groups.\\ 
Lastly, in order to decompose these new characters, we need to study the restriction of character sheaves in the principal series in more details. We show a formula (see Proposition \ref{prop_formula}) which could be of independent interest, as it gives information on certain character sheaves, and hence their characteristic functions on mixed conjugacy classes.\\

The assumption that $p$ is good is crucial since we do use a lot of properties for the Generalised Gelfand--Graev characters which are not yet proven for the extension to bad primes as defined by Geck \cite{geckGeneralisedGelfandGraev2021}.
Since we do not know yet a basic set for the unipotent $\ell$-blocks for groups with a non-trivial center, we cannot extend our result to any finite reductive group of exceptional type.

\subsubsection*{Acknowledgments} I would like to thank my PhD advisors Gunter Malle and Olivier Dudas for their suggestions and helpful advice, as well as Jay Taylor for always kindly answering my questions. 
This work was financially supported by the SFB-TRR 195 of the German Research Foundation (DFG).

\subsection*{Notation}

We fix $\ell$ a prime number. 
\subsubsection*{On finite groups and their representation theory}

Let $A$ be a finite group and $\F$ a field. We denote by $\irr_\F(A)$ the set of isomorphisms classes of irreducible $\F A$-modules of the group $A$. For any $\F A$-modules $V, W$, we let \[\langle V, W \rangle_\F \defeq \dim_\F\Hom_{\F A}(V,W).\] 
To denote the ordinary irreducible characters of the group $A$, we will use $\irr(A)$. 
%
We fix a splitting $\ell$-modular system $\left(\Oo,\K,\k\right)$, as follows:
\begin{itemize}
	\item $\Oo$ is  a complete discrete valuation ring of characteristic $0$ with a unique maximal ideal $M$,
	\item $\K$ is the field of fraction of $\Oo$, also of characteristic zero. We assume that $\K$ is big enough for the group $A$ we are considering, that is it contains all $|A|$th roots of unity. In particular, with respect to an inclusion $\K \subseteq \C$, we have $\irr_\K(A) \cong \irr_\C(A)$ via scalar extension. 
	\item $\k = \Oo/M$ is a field of characteristic $\ell$. We will assume $\k = \overline{\F_\ell}.$
\end{itemize}
Each $W \in \irr_\k(A)$ has a unique projective indecomposable cover $P_W$.
Recall that to any projective indecomposable $\k A$-module $P$ corresponds a projective  $\Oo A$-module $P^{\Oo}$ such that $P^\Oo \otimes_{\Oo} \k \cong P$, unique up to isomorphism. On the other hand, to any $\K A$-module $V$ corresponds at least one $\Oo A$-module $V_\Oo$, free over $\Oo$ such that $V_\Oo \otimes_\Oo \K \cong V$. 
For any $\K A$-module $V$ and any projective indecomposable $\k A$-module $P$, we have by Brauer reciprocity,
\[\langle P,V_\Oo \otimes_{\Oo} \k\rangle_{\k} = \langle P^\Oo \otimes_\Oo \K, V\rangle_{\K} \eqqcolon [P,V].\]

We denote the decomposition matrix of $A$ by $D^A = (d^A_{V,W})_{V \in \irr_\K(A),W \in \irr_\k(A)}$ with entries $$d^A_{V,W} \defeq [P_W,V].$$

For $P$ a projective $\k G$-module, let $\Psi_P$ denotes the character associated to the module $P^\Oo \otimes_\Oo \K$. We may sometimes write $d_{\psi_V,\Psi_P} = [P_W,V] = \langle \Psi_P, \psi_V \rangle$, where $\psi_V$ is the character of the irreducible $\K A$-module $V$. \\

We define $\mathcal{M}(A)$ as the set of $A$-conjugacy classes of pairs $[a,\phi]$ with $a \in A$ and $\phi \in \irr(C_A(a))$. We also define a pairing coming from \cite[$\S$ 4]{lusztigClassIrreducibleRepresentations1979}:
\begin{align*}
	\{ \, , \,\} : \mathcal{M}(A) \times \mathcal{M}(A) &\to \C,\\
	\left([a,\phi],[b,\psi]\right) &\mapsto \frac{1}{|C_A(a)||C_A(b)|}\sum_{g \in A, a \in C_A(gbg^{-1})}\phi(gbg^{-1})\psi(g^{-1}a^{-1}g). 
\end{align*}

Let $\tilde{A}$ be another finite group such that $A$ is a normal subgroup of $\tilde{A}$ and $\tilde{A}/A$ is cyclic of order $c$ with a generator $aA$ for some $a \in\tilde{A}$. 
The set $\overline{\Mm}(A \subseteq \tilde{A})$ consists of all $\tilde{A}$-conjugacy classes of pairs $(b,\sigma)$ such that $b \in aA$ and $\sigma \in \irr(C_A(b))$. 
We denote by $\Mm^\ell(A)$ and $\overline{\Mm}^\ell(A \subseteq \tilde{A})$ the same sets as above where instead of considering ordinary irreducible characters of the centralizers, we consider irreducible $\k$-modules of the same centralizers. 
\subsubsection*{On finite redutive groups}

We fix $\G$ a connected reductive group over ${k}$, an algebraically closed field of characteristic $p > 0$. Let $F : \G \to \G$ be a Frobenius endomorphism endowing $\G$ with an $\F_q$-structure, for $q$ a power of the prime $p$. We write $G \defeq \G^F$. \\

We fix $\T$ a maximally split torus of $\G$ contained in an $F$-stable Borel subgroup $\B$  with unipotent radical $\textbf{U}$.

We denote by $ \Phi^+ \subseteq \Phi$ the set of (positive) roots determined by $\T$ and $\B$. We set $\Delta = \{\alpha_1,\dots,\alpha_n\}$ the set of simple roots of $\Phi$. 
To each root $\alpha \in \Phi$ corresponds a $1$-dimensional root subgroup $\U_\alpha$ of $\G$ and a root subspace $\mathfrak{g}_\alpha$ of the Lie algebra $\mathfrak{g}$ of $\G$. We also write $X(\T)$ for characters and $Y(\T)$ for the cocharacters of $\T$.\\
The Weyl group of $\G$ associated to $\T$ will be denoted by $W \defeq N_\G(\T)/\T$. For any root $\alpha \in \Phi$, we set $s_\alpha \in W$ the corresponding reflection. For each $w \in W$, we fix a representative $\dot{w} \in N_\textbf{G}(\T)$. We will often abuse notation and write $w$ instead of $\dot{w}$. \\ 
We write $\G^*$ for the dual group of $\G$ with corresponding Frobenius map $F^*$, and $W^*$ for the corresponding Weyl group. \\
We denote by $\sigma$ the automorphism of $W$ induced by $F$. 
\begin{hyp}
We always assume $\sigma$ to be ordinary as defined in \cite[3.1]{lusztigCharactersReductiveGroups1984}.
\end{hyp}

For $g \in \G$, we write $g_s$ for its semisimple part and $g_u$ for its unipotent part. More generally, for any subset $J$ of $\G$, we denote by $J_s$ the set of semisimple parts of elements in $J$ and $J_u$ the set of unipotent parts of elements in $J$. For any algebraic group $\H$, we write $\H_{uni}$ for its unipotent variety consisting of the unipotent elements of $\H$. \\
For $u \in \G$ unipotent, we define $A_\G(u) \defeq C_\G(u)/C^\circ_\G(u)$. We fix $X_\G$ the set of all unipotent conjugacy classes of the connected reductive group $\G$. We denote  by $u_\Cc$ any $F$-stable element of $\Cc$ such that $A_\G(u_\Cc) = A_\G(u_\Cc)^F$. If the center $Z(\G)$ of $\G$ is connected and $\G/Z(\G)$ is simple, such an element $u_\Cc$ will always exists for any $\Cc \in X_\G$ by \cite[Prop. 2.4]{taylorUnipotentSupportsReductive2013}.

We will use the notation from CHEVIE (\cite{michelDevelopmentVersionCHEVIE2015}) for the names of unipotent classes and unipotent characters.

\section{Counting the number of modular representations of $G$}
In order to understand the $\ell$-decomposition matrix of $G$, not only do we need the number $n \defeq |\irr_\k(G)|$ of columns of the matrix, but also a labeling of the columns and the rows. More precisely, using the parametrisation of the ordinary modules of $G$, we choose a subset of them, say $\mathcal{B} \subseteq \irr_\K(G)$ such that $|\mathcal{B}| = n$. Our hope is that to each $V \in \mathcal{B}$, we can associate  a projective indecomposable $\k G$-module $\P_V$, such that the matrix $(\langle W, P_V^\Oo \otimes_\Oo K\rangle_K)_{W,V \in \mathcal{B}}$ is lower-unitriangular. In other terms, we want to find \textbf{unitriangular basic sets}. In this section, we will fix a candidate set $\mathcal{B}$ in several cases. 
\begin{hyps}
In this section, we always assume that $Z(\G)$ is connected and $\G/Z(\G)$ is simple. Moreover, the prime $p$ is different from $\ell$ and $p$ is good for $\G$. 
\end{hyps}
\subsection{Parametrisation of the ordinary characters of $G$}
In \cite{lusztigCharactersReductiveGroups1984}, Lusztig gave a parametrisation of the ordinary irreducible characters of $G$. We very briefly summarise these results. \\

One can partition the irreducible characters of $G$ into series indexed by semisimple elements of the dual group (\cite[\nopp7.6]{lusztigIrreducibleRepresentationsFinite1977}), called \textbf{Lusztig series}: \[\irr(G) = \bigsqcup_s \mathcal{E}(G, s),\] where $s$ runs over representatives of the conjugacy classes of semisimple elements of $(\G^*)^{F^*}$. The \textbf{unipotent} series $\Uch(G) \defeq \Ee(G,1)$ is of particular interest, thanks to the Jordan decomposition of characters. Indeed, by \cite[Thm.~4.23]{lusztigCharactersReductiveGroups1984}, there is a bijection \[\Uch(C_{\G^*}(s)^{F^*}) \longleftrightarrow \Ee(G, s) \quad \forall s \in (\G^*)^{F^*}. \]
The set $\Uch(G)$ can itself be decomposed into families which could be defined through the notion of unipotent support, a unique unipotent conjugacy class of $\G$ associated to each irreducible character.
Thanks to \cite[Prop.~4.2 and Cor.~5.2]{geckExistenceUnipotentSupport2000a}, we say that two unipotent characters are in the same \textbf{family} if and only if they have the same unipotent support.\\
As a consequence, for any semisimple element $s \in (\G^*)^{F^*}$, we have a partition
\[\mathcal{E}(G,s) = \bigsqcup_\Ff \mathcal{E}(G,s)_\Ff,\]
where $\Ff$ runs over the families of $\Uch(C_{\G^*}(s)^{F^*})$. We say that a unipotent class in $\G$ is \textbf{special} if it is the unipotent support of a unipotent character. More generally, we call $g \in \G^*$ \textbf{special} if $(g_u)_{C_{\G^*}(g_s)}$ is the unipotent support of a unipotent character of $C_{\G^*}(g_s)$, with $g_s \in (\G^*)^{F^*}$. If $\Ff$ is the family of $\Uch(C_{\G^*}(g_s)^{F^*})$ with unipotent support $(g_u)_{\G^*}$, we set \[\irr(G)_g \defeq \mathcal{E}(G,g_s)_\Ff,\]
under Jordan decomposition of characters. 

Lusztig defined a surjective map  $\Phi$ from the special $\G^*$-conjugacy classes to the unipotent conjugacy classes of $\G$. We recall a few notions needed for this definition. 
We define $\mathcal{N}_\G$ to be the set of all pairs $(\Cc,\phi)$ with $\Cc \in X_\G$ and $\phi \in \irr(A_\G(u))$. 
The \textbf{Springer correspondence} gives an injective map $i_\G : \irr(W) \to \mathcal{N}_\G$, see \cite{springerTrigonometricSumsGreen1976}, \cite{lusztigIntersectionCohomologyComplexes1984}.\\
 Recall that for $\psi \in \irr(W)$, the \textbf{$b$-invariant} of $\psi$ is defined as the smallest non-negative integer $n \in \N$ such that $\psi$ occurs in the character of the $n$th symmetric power of the natural representation of $W$ (\cite[$\S$ 2]{lusztigClassIrreducibleRepresentations1979}). If $W'$ is a subgroup of $W$ generated by reflections, then for each $\psi' \in \irr(W')$ there exists a unique $\psi \in \irr(W)$ such that $\langle\psi ,\Ind^W_{W'}(\psi')\rangle = 1$ and the $b$-invariants of $\psi$ and $\psi'$ agree. It is called the \textbf{j-induction} of $\psi'$ and denoted $j^W_{W'}(\psi')$ (\cite[$\S$ 3]{lusztigClassIrreducibleRepresentations1979}). For $\psi \in \irr(W)$, the \textbf{$a$-invariant} is defined as the largest $n \in \N$ such that $\textbf{q}^n$ divides the generic degree of $\psi$, for an indeterminate $\textbf{q}$ (\cite[$\S$ 2]{lusztigClassIrreducibleRepresentations1979}). We say that $\psi \in \irr(W)$ is \textbf{special} if $a_\psi = b_\psi$.  \textbf{Lusztig's map} $\Phi$ from the special conjugacy classes of $\G^*$ to $X_\G$ is then constructed as follows.
 \begin{enumerate}
 	\item We start with $g = sv \in \G^*$ special where we assume that $s \in \T^*$ and $v \in C_{\G^*}(s)$ is unipotent. 
 	\item Since $(v)_{C_{\G^*}(s)}$ is special, there is a unique special $\psi' \in \irr(W_{C_{\G^*}(s)})$ such that $i_{C_{\G^*}(s)}(\psi') = ((v)_{C^\circ_{\G^*}(s)},1)$.
 	\item We apply j-induction and construct a character $j^{W^*}_{W_{C_{\G^*}(s)}}(\psi')$.
 	\item  Since $W \cong W^*$, there is an associated character $\psi$ of $W$ which corresponds to $j^{W^*}_{W_{C_{\G^*}(s)}}(\psi')$.
 	\item  Lastly, we apply Springer correspondence $i_\G(\psi) = (\Cc, \phi) \in \mathcal{N}_\G$. We then set $ \Phi(g) = \Cc$.
 \end{enumerate}
 
 By \cite[Thm.~10.7]{lusztigUnipotentSupportIrreducible1992} for $\rho \in \irr(G)$ if $g \in \G^*$ is special such that $\rho \in \irr(G)_g$, then $\Phi(g)$ is the unipotent support of $\rho$. \\
 
 Lastly, to each family $\Ff$ of $\Uch(G)$ with unipotent support $\Cc \in X_\G$, Lusztig defined an \textbf{ordinary canonical quotient} $\Omega_{\G,\Cc}$, a certain quotient of $A_\G(u_\Cc)$ (\cite[Thm.~4.23 and $\S$ 13.1.3]{lusztigCharactersReductiveGroups1984} and \cite{lusztigFamiliesSpringerCorrespondence2014}). He showed that there exists a finite group $\tilde{\Omega}_{\G,\Cc}$ such that $\Omega_{\G,\Cc}$ is a normal subgroup of $\tilde{\Omega}_{\G,\Cc}$, $|\tilde{\Omega}_{\G,\Cc}:\Omega_{\G,\Cc}| = c$, where $c$ is the order of the automorphism $\sigma$ and such that there is a bijection between $\overline{\mathcal{M}}(\Omega_{\G,\Cc} \subseteq \tilde{\Omega}_{\G,\Cc})$ and $\Ff$.

\subsection{Basic sets and their parametrisation}

\begin{hyp}
	We further assume in this subsection that $\ell$ does not divide the order of the largest quotient of $Z(\G)$ on which $F$ acts trivially.
\end{hyp}

We only consider the \textbf{unipotent blocks}, which can be defined thanks to \cite[Thm.~2.2]{broueBlocsSeriesLusztig1989} as 
\[ B_1(G) \defeq \bigsqcup \mathcal{E}(G,s),\] where $s$ runs over representatives of conjugacy classes of semisimple $\ell$-elements of $(\G^*)^{F^*}$. These unipotent blocks are particularly important as most other unions of blocks are Morita equivalent to them, thanks to Bonnafé and Rouquier \cite[Thm.~11.8]{bonnafeCategoriesDeriveesVarietes2003}. \\

When $\ell$ is good, we know a basic set of ordinary irreducible characters for the unipotent blocks.
\begin{thm}[{\cite[Thm.~5.1]{geckBasicSetsBrauer1991}}]
	Recall that $p \neq \ell$ and $\ell$ are good for $\G$ and $Z(\G)$ is connected. Then, the number of irreducible Brauer characters in the unipotent blocks $B_1(G)$ is \[|\Ee(G,1)| = \sum_{\Cc}|\overline{\mathcal{M}}(\Omega_{\G,\Cc} \subseteq \tilde{\Omega}_{\G,\Cc})|,\]
	where $\Cc$ runs over the $F$-stable special unipotent conjugacy classes of $\G$.
\end{thm}
%

In the case where $\ell$ is bad, $\Ee(G,1)$ does not give a basic set for $B_1(G)$. However, Chaneb (\cite{chanebBasicSetsUnipotent2021}) found another parametrisation, which involves more unipotent conjugacy classes in the sum.\\

Recall that $t \in \G$ a semisimple element is called \textbf{quasi-isolated} if $C_\G(t)$ is not included in a proper Levi subgroup of $\G$. If moreover, $C_\G^\circ(t)$ is not contained in a proper Levi subgroup of $\G$, then we say that $t$ is \textbf{isolated}.
\begin{dfn}[Chaneb]
	We say that a unipotent class $\Cc \in X_G$ is \textbf{$\ell$-special} if there exists $s \in \G^*$ an isolated semisimple $\ell$-element and $v \in C_{\G^*}(s)$ unipotent such that $\Phi(sv) = \Cc$.
\end{dfn}  In particular, any special unipotent class of $\G$ is $\ell$-special. 
 Analogously to Luzstig, Chaneb defined another quotient of $A_\G(u_\Cc)$, the \textbf{$\ell$-canonical quotient} $\Omega^\ell_{\G,\Cc}$ associated to an $\ell$-special class $\Cc$. Note that when $\ell$ is good and $\G$ is simple and adjoint, then $A_\G(u_\Cc)$ is an $\ell'$-group for all special unipotent conjugacy classes $\Cc$ of $\G$ and $\Omega^\ell_{\G,\Cc} = \Omega_{\G,\Cc}$. 
 This new definition gave the following result: 
 
 \begin{thm}[{\cite[Thm.~3.16]{chanebBasicSetsUnipotent2021}}]
 Assume that $\G$ is simple adjoint of exceptional type. There exists a finite group $\tilde{\Omega}^\ell_{\G,\Cc}$ such that $\Omega^\ell_{\G,\Cc}$ is a normal subgroup of $\tilde{\Omega}^\ell_{\G,\Cc}$, $|\tilde{\Omega}^\ell_{\G,\Cc}:\Omega^\ell_{\G,\Cc}| = c$, where $c$ is the order of the automorphism $\sigma$. Then, the number of irreducible Brauer characters in the unipotent blocks $B_1(G)$ is \[ \sum_{\Cc}|\overline{\mathcal{M}}(\Omega^\ell_{\G,\Cc} \subseteq \tilde{\Omega}^\ell_{\G,\Cc})|,\]
 where $\Cc$ runs over the $F$-stable $\ell$-special unipotent conjugacy classes of $\G$.
 \end{thm}


For each $F$-stable $\ell$-special unipotent class $\Cc$, we write $\alpha^\ell_\Cc \defeq |\overline{\mathcal{M}}^\ell(\Omega^\ell_{G,\Cc} \subseteq \tilde{\Omega}^\ell_{\G,\Cc})|.$ Recall that when $p$ is good and $\G$ simple adjoint of exceptional type all unipotent conjugacy classes of $\G$ are $F$-stable \cite[Section 5.1]{hezardSupportUnipotentFaisceauxcaracteres2004}.

\section{The general approach for finding a unitriangular basic set}
We would like to show that the decomposition matrix of the unipotent $\ell$-blocks of $G$ is lower unitriangular. Firstly, as the following result shows, it is enough to find suitable projective $\k G$-modules, not necessarily indecomposable.

\begin{prop}[{\cite[Lem.~2.6]{geckBasicSetsBrauer1994}}]\label{prop_projsuff} Let $A$ be a finite group. Let $B$ be a union of $\ell$-blocks of the group $A$ and $n \defeq |\irr_\k(B)|$. Assume that there exist irreducible $\K A$-modules $V_1, \dots, V_n$ in  $B$ and projective $\k A$-modules $P_1,\dots,P_n$ such that the matrix $([  V_i,P_j])_{1\leq i,j\leq n}$ is lower unitriangular. Then the $\ell$-decomposition matrix of $B$ is unitriangular. 
\end{prop}
%

We will proceed as follows.
\begin{enumerate}
	\item Fix a total ordering of the $\ell$-special unipotent conjugacy classes $\Cc_1 < \dots < \Cc_r$.
	\item For  each $\ell$-special class $\Cc_n$ with $1 \leq n \leq r$,
	\begin{itemize}
	\item find a set of $\alpha^\ell_{\Cc_n}$ projective $ \k G$-modules $\{P^n_i \mid 1 \leq i \leq \alpha^\ell_{\Cc_n}\}$
\item and a set of  $\alpha^\ell_{\Cc_n}$ irreducible $\K G$-modules $\{V^n_i \mid 1 \leq i \leq \alpha^\ell_{\Cc_n}\}$,
	\end{itemize}

such that 
\begin{enumerate}
	\item[(A)]  $([V_i^n, P_j^n])_{1\leq i,j\leq \alpha^\ell_{\Cc_n}}$ is lower unitriangular,
\item[(B)] and for all $1 \leq m \leq r$, if $m < n$, $[  V^m_i, P^n_j] = 0$ for all $1 \leq i \leq  \alpha^\ell_{\Cc_m}$ and $1 \leq j \leq  \alpha^\ell_{\Cc_n}$.
\end{enumerate}
\end{enumerate}
The total ordering will be based on the dimensions of the unipotent classes. Besides, we will choose as $\K G$-modules associated to a certain $\ell$-special unipotent class $\Cc$ a subset of the ones which have wave front set $\Cc$ (a dual notion to the unipotent support). Lastly, for the projective $\k G$-modules we will consider Kawanaka characters which decompose generalised Gelfand--Graev characters. \\
We first recall the main results of \cite{brunatUnitriangularShapeDecomposition2020} leading to the proof of their Theorem~A. Most of their intermediate results do not depend on the prime~$\ell$. 

\begin{hyp}
	In this section, we assume that $p$ is good for $\G$. 
\end{hyp}
\subsection{Generalised Gelfand--Graev characters}
\subsubsection{Unipotent conjugacy classes and nilpotent orbits}

We parameterize the unipotent conjugacy classes of $\G$. Firstly we introduce some notation, following \cite[Chapter 5]{carterFiniteGroupsLie1985}. 
\begin{dfn}
	Let $\lambda \in Y({\T})$ be a cocharacter of $\T$. We define the following subgroups of $\G$:
	\begin{align*}
	{\P}_\lambda &\defeq \langle {\T}, {\U}_\alpha \mid \alpha \in \Phi \text{ with } \langle \alpha, \lambda\rangle \geq 0 \rangle,\\
	{\L}_\lambda &\defeq \langle {\T}, {\U}_\alpha \mid\alpha \in \Phi \text{ with } \langle \alpha, \lambda\rangle = 0\rangle,\\
	{\U}_\lambda &\defeq \langle {\U}_\alpha \mid \alpha \in \Phi \text{ with } \langle \alpha, \lambda\rangle > 0\rangle.	
	\end{align*}
	Observe that ${\P}_\lambda$ is a parabolic subgroup of $\G$ with Levi subgroup ${\L}_\lambda$ and unipotent radical ${\U}_\lambda$.
	For any integer $i >0$ we also define \begin{align*}
	{\U}_\lambda(i) &\defeq \langle {\U}_\alpha \mid\alpha \in \Phi^+ \text{ with } \langle \alpha, \lambda\rangle \geq i\rangle,\\
	{\U}_\lambda(-i) &\defeq \langle {\U}_\alpha \mid\alpha \in \Phi^+ \text{ with } \langle \alpha, \lambda\rangle \leq -i\rangle.	
	\end{align*}
	Observe that ${\U}_\lambda(-i) = {\U}_{-\lambda}(i)$.
\end{dfn}

We set $F_q: \overline{\F_p} \to \overline{\F_p}$ the standard Frobenius endomorphism defined by $F_q(x) = x^q$ for $x\in  \overline{\F_p}$.
We define an $F$-action on $Y({\T})$ as follows: $F \cdot \lambda = F \circ \lambda \circ F^{-1}_q$.

\subsubsection{Step 1: Parameterizing unipotent conjugacy classes reduces to parameterizing nilpotent orbits}

We denote by $\mathcal{N}$ the variety of nilpotent elements of the Lie algebra $\mathfrak{g}$ associated to $\G$ and by $\mathcal{U}$ the variety of unipotent elements of $\G$. By \cite[$\S$ 10]{mcninchOptimalSLHomomorphisms2005}, we have

\begin{prop}[Springer, Serre] There exists a homeomorphism of varieties $\Psi_{spr}: \mathcal{U} \to \mathcal{N}$ such that for all elements $g \in\G$ and unipotent elements $u\in \mathcal{U}$, we have 
 \[\Psi_{\text{spr}}({}^gu) = \Ad(g)(\Psi(u)).\]
 The induced map between the unipotent conjugacy classes of $\G$ and the nilpotent orbits of $\g$ does not depend on the choice of $\Psi_{spr}$. 
\end{prop}
This map is called the \textbf{Springer homeomorphism}. If the group is proximate (\cite[Def.~2.10]{taylorGeneralizedGelfandGraevRepresentations2016}) then any Springer homeomorphism is an isomorphism of varieties. (\cite[Lem.~3.4]{taylorGeneralizedGelfandGraevRepresentations2016})

%

\subsubsection{Step 2: Nilpotent orbits are parametrized by weighted Dynkin diagrams}
Recall that we assume that $p$ is good for $\G$. We let $\G_\C$ be a reductive group defined over $\C$ with Borel subgroup $\B_\C$ and maximal torus $\T_\C$ such that it defines an isomorphic root datum to the one associated to $(\G, \B,\T)$.
To each non-zero nilpotent orbit $\mathcal{O}$, one can associate an $\mathfrak{sl}_2$-triple $\{e,f,h\} \subseteq \g_\C= \text{Lie}(\G_\C)$ such that $e \in \mathcal{O}$, by the Jacobson--Morozov Theorem. We may further assume that $\alpha(h) \geq 0$ for all simple roots $\alpha$. 
We then define the \textbf{weighted Dynkin diagram} associated to $\mathcal{O}$ as the map $d_\mathcal{O} : \Delta \to \Z$, $d_\mathcal{O}(\alpha) = \alpha(h)$, that we extend linearly to a map on all roots of $\G$. 	The weighted Dynkin diagram $d_\mathcal{O}$ defined above does not depend on the choice of $\{e,f,h\}$ up to conjugation. Moreover, two nilpotent orbits $\mathcal{O}$ and $\mathcal{O}'$ have the same weighted Dynkin diagram if and only if they are the same. Lastly, there is a unique cocharacter $\lambda_{d_\mathcal{O}} \in Y({\T})$ such that for all roots $\alpha$ \[d_\mathcal{O}(\alpha) = \langle \alpha, \lambda_{d_\mathcal{O}}\rangle,\] 
for the pairing between $X(\T)$ and $Y(\T)$.
We write $\mathcal{D}$ for the set of all the weighted Dynkin diagrams constructed as above. We also define
\[Y^\G_\mathcal{D} \defeq \{{}^g\lambda_d \mid 
d \in \mathcal{D}, \, g \in \G\}.\]
Lastly for $u \in \mathcal{U}$ we define $Y^\G_\mathcal{D}(u)$ the subset of $\lambda \in Y^\G_\mathcal{D}$ is such that $\Psi_{\text{spr}}(u)$ is in the unique dense open ${\L}_\lambda$-orbit of $\Lie({\U}_\lambda(2)\backslash{\U}_\lambda(3))$. 

%

\subsubsection{Definition of the generalised Gelfand--Graev characters}

We now recall the construction of the generalised Gelfand--Graev characters (GGGC's) following the notation in \cite[Section II.6]{brunatUnitriangularShapeDecomposition2020}. These characters were first defined in \cite{kawanakaGeneralizedGelfandGraevRepresentations1986}, and another construction was given in \cite{taylorGeneralizedGelfandGraevRepresentations2016}.

\begin{hyp}
	From now on, we assume that $\G$ is proximate.
\end{hyp}

We fix a Kawanaka datum $\mathcal{K} =(\Psi_{\text{spr}},\kappa,\chi_p)$ as in \cite[Def.~6.1, Lem.~6.3]{brunatUnitriangularShapeDecomposition2020}. For each $u \in\mathcal{U}^F$  a rational unipotent element and $\lambda \in Y^\G_\mathcal{D}(u)^F$, we fix a certain irreducible character  $\xi^G_{u,\lambda}$ of ${\U}_\lambda(-1)^F$ (\cite[Eq.~6.4]{brunatUnitriangularShapeDecomposition2020}). We can associate to $\xi^G_{u,\lambda}$ an irreducible $\k G$-module as ${\U}_\lambda(-1)^F$ is a $p$-group, whence an $\ell'$-group. Moreover, by \cite[Eq.~6.5]{brunatUnitriangularShapeDecomposition2020}, for any $x \in G$,\[{}^x\xi^G_{u,\lambda} = \xi^G_{{}^xu,{}^x\lambda}.\]

\begin{dfn}[Kawanaka]
For $u \in\mathcal{U}^F$ a rational unipotent element and $\lambda \in Y^\G_\mathcal{D}(u)^F$ an $F$-stable cocharacter, we define the \textbf{generalised Gelfand--Graev character} (GGGC) of $G$ as \[\gamma^G_u \defeq \Ind^{G}_{{\U}_\lambda(-1)^F}(\xi^G_{u,\lambda}).\]
\end{dfn}

One can show that $\gamma^G_u$ does not depend on the choice of cocharacter $\lambda \in Y^\G_\mathcal{D}(u)^F$ (\cite[below Def.~6.6]{brunatUnitriangularShapeDecomposition2020}). Moreover, for any $x \in G$, 
\[\gamma^G_u = \gamma^G_{{}^xu}. \]

In particular, for an $F$-stable unipotent conjugacy class $\Cc$, one could obtain at most as many different generalised Gelfand--Graev characters of the form $\gamma^G_u$ for some $u \in \Cc^F$ as there are conjugacy classes of $A_\G(u_\Cc)$.\\ 
Observe as well that since ${\U}_\lambda(-1)^F$ is a $p$-group, whence an $\ell'$-group and since induction preserves projectivity, we can associate to each GGGC a projective $\k G$-module. In other words, there exists a projective $\k G$-module $\Gamma^G_u$ such that $\gamma_u^G$ is the character associated to $(\Gamma_u^G)^\Oo \otimes_\Oo \K$.  

\subsubsection{Wave front set}
We define a dual concept to the unipotent support using GGGCs. 
\begin{dfn}
 Let $\rho  \in \irr(G)$. A \textbf{wave front set} of $\rho$ is an $F$-stable unipotent conjugacy class $\Cc$ of $\G$ such that:
 \begin{enumerate}
\item there is $v  \in \Cc^F$ such that $\langle \gamma^G_v,\rho\rangle \neq 0$ and 
\item for any unipotent conjugacy class $\Cc'$ of $\G$ such that $\langle \gamma^G_{v'},\rho \rangle \neq 0$ for some $v' \in \Cc'$, we have $\dim(\Cc') \leq \dim(\Cc)$. 
 \end{enumerate}
\end{dfn}
Similarly to the unipotent support, the wave front set is in fact unique.

\begin{thm}[{\cite[Thm.~14.10, Thm.~15.2]{taylorGeneralizedGelfandGraevRepresentations2016}}] \label{thm_unicitywave}
 Let $\rho  \in \irr(G)$. Then $\rho$ has a unique wave front set, which we denote by $\Cc_\rho^*$. Moreover, for any unipotent element $u \in G,$ if $\langle \gamma_u^\G, \rho \rangle \neq 0$, then $(u)_\G \subseteq \overline{\Cc_\rho^*}.$
\end{thm}
For an irreducible character $\rho \in \irr(G)$, we write $$\rho^* \defeq \pm D_\G(\rho),$$ where the sign is the unique choice making the Alvis--Curtis dual $D_\G(\rho)$ an irreducible character of $G$, see \cite[Def.~3.4.1]{geckCharacterTheoryFinite2020} for a definition of $D_\G$.\\
Unipotent supports and wave front sets are deeply linked:

\begin{lem}[{\cite[Lem.~14.15]{taylorGeneralizedGelfandGraevRepresentations2016}}]
	Let $\rho \in \irr(G)$. Then the unipotent support $\Cc_{\rho^*}$ of $\rho^*$ is the wave front set $C^*_\rho$ of $\rho$, (and conversely the unipotent support of $\rho$ is the wave front set of $\rho^*$). 
\end{lem}

\begin{rmk} \label{rmk_plan2}
Let us look at our plan we explained at the beginning of this section. We choose for a total ordering of the $\ell$-special unipotent conjugacy classes one such that $\Cc_i < \Cc_j$ if $\dim \Cc_i \leq \dim \Cc_j$ for all $1 \leq i < j \leq r$. For each $\Cc_n$, in order to satisfy Condition (B), we would like to choose irreducible ordinary representations of $G$ with wave front set $\Cc_n$ and GGGCs of the form $\Gamma_u^G$ for $u \in \Cc_n$. However, if the number of conjugacy classes of $A_\G(u_\Cc)$ is smaller than $\alpha^\ell_{\Cc_n}$, we do not have enough projective $\k G$-modules.
\end{rmk}
%
%
%

\subsection{Admissible coverings and Kawanaka characters} To overcome the difficulty that there might not be enough generalised Gelfand--Graev characters, Brunat--Dudas--Taylor decompose the GGGCs into a direct sum of other characters called Kawanaka characters. To do so, they first define a lift of each ordinary canonical quotient. 

\subsubsection{Definition and existence of an admissible covering}
\begin{dfn}[{\cite[Def.~7.1]{brunatUnitriangularShapeDecomposition2020}}]
	Let $u \in \mathcal{U}^F$ be a rational unipotent element. Let $A \leq C_\G(u)$ be a subgroup and $\lambda \in Y^\G_\mathcal{D}(u)^F$ be an $F$-stable cocharacter. We say that the pair $(A, \lambda)$ is \textbf{admissible} for $u$  if the following hold:
	\begin{enumerate}
\item the subgroup $A \subseteq \textbf{L}^F_\lambda$,
\item the subgroup $A$ contains only semisimple elements,
\item and for all $a \in A$, we have $a \in C^\circ_{\textbf{L}_\lambda}(C_A(a)).$
	\end{enumerate}
If $\bar{A}$ is a quotient of $A_G(u)$ on which $F$ acts, we say that the pair $(A,\lambda)$ is an \textbf{admissible covering} for $\bar{A}$ if:
\begin{enumerate}
	\item[(4)] the restriction of the map $C_\G(u) \to \bar{A}$ to $A \to \bar{A}$ fits into the following short exact sequence \[1 \longrightarrow Z \longrightarrow A \longrightarrow \bar{A} \longrightarrow  1\] where $Z \leq Z(A)$ is a central subgroup with $Z \cap [A,A] =\{ 1 \}.$
\end{enumerate}
\end{dfn}

\begin{prop}[{\cite[Sections 9 and 10]{brunatUnitriangularShapeDecomposition2020}}] \label{prop_admcov} Assume that $\G$ is simple and adjoint. Let $\Cc$ be a special unipotent conjugacy class of $\G$. Then there always exist
	\begin{itemize}
		\item an $F$-stable unipotent element $u_\Cc \in \Cc$ such that $F$ acts trivially on $A_\G(u_\Cc)$ and
		\item an admissible pair $(A_\Cc,\lambda)$ for $u_\Cc$ which is an admissible covering of $\Omega_{\G,\Cc}$, and such that $A_\Cc$ is abelian or $A_\Cc \cong \Omega_{\G,\Cc}$ .
	\end{itemize}Moreover $|A_\Cc|$ is divisible only by bad primes for $\G$.
\end{prop}

We describe in more details the case of exceptional groups. 
\begin{prop}[{ \cite[Section 10]{brunatUnitriangularShapeDecomposition2020}}]\label{admcovexcp} Assume that $\G$ is simple adjoint of exceptional type. Let $\Cc$ be a special unipotent conjugacy class of $\G$.
	We distinguish between the following cases: 
	\begin{enumerate}
		\item If $\Omega_{\G,\Cc}$ is trivial, then we choose $A_\Cc = \{1\} \subseteq \G$ for an admissible covering.
		\item If $\G$ is of type $E_8$ and $\Cc = E_8(b_6)$, then $A_\G(u_\Cc) \cong S_3$ and $\Omega_{\G,\Cc} \cong A_\Cc \cong S_2$.
		\item If $\G$ is of type $E_7$ and $\Cc = A_4 + A_1$ or $\G$ is of type $E_8$ and $\Cc$ is one of $E_6(a_1) +A_1, D_7(a_2), A_4 + A_1$, then $A_\G(u_\Cc) \cong \Omega_{\G,\Cc} \cong S_2$ and $ A_\Cc \cong C_4$.
		\item Else, $\Omega_{\G,\Cc}$ is not trivial and $A_\G(u_\Cc) \cong \Omega_{\G,\Cc} \cong A_\Cc$.
	\end{enumerate}
\end{prop}

 \subsubsection{Definition of a Kawanaka character}

	We fix $u \in \mathcal{U}^F$ a rational unipotent element and $(A,\lambda)$ be an admissible pair for $u$. To each $a \in A$, one can find an element $u_a \in ((u)_\G)^F$ such that the set $\{u_a \mid a \in A\}$ satisfies certain conditions \cite[Lem.~7.6]{brunatUnitriangularShapeDecomposition2020}. It is called a set of \textbf{admissible representatives}.


For any $a \in A$, we define an extension $\tilde{\xi}^G_{u_a,\lambda} \in \irr(C_A(a) \ltimes \textbf{U}_\lambda(-1)^F)$ of $\xi^G_{u_a,\lambda}$ as in \cite[Section 7.4]{brunatUnitriangularShapeDecomposition2020}. It is called the \textbf{Weyl extension} and is well-understood, by \cite[Lem.~7.11]{brunatUnitriangularShapeDecomposition2020}, \cite{gerardinWeilRepresentationsAssociated1977}. 
%
 
 \begin{dfn}[{\cite[Def.~7.13]{brunatUnitriangularShapeDecomposition2020}}]\label{dfn_Kawchar}
Let $a \in A$ and $\phi \in \irr(C_A(a))$. We define the \textbf{Kawanaka character} associated to the pair $(a, \phi)$ to be \[K^G_{(a, \phi)} \defeq \Ind^{G}_{C_A(a) \ltimes \textbf{U}_\lambda(-1)^F}\left( \tilde{\xi}^G_{u_a,\lambda} \otimes \Inf^{C_A(a) \ltimes \textbf{U}_\lambda(-1)^F}_{C_A(a)}\phi\right). \]
 \end{dfn}
 
We observe that 
 \begin{align}\label{KawGGGr}
 	\gamma^G_{u_a} = \sum_{\phi \in \irr(C_A(a))} \phi(1)K^G_{(a,\phi)}.
 \end{align}
 
 Moreover, for any $a,b \in A$ and $\phi \in \irr(C_A(a))$, we have $K^G_{({}^ba,{}^b\phi)} = K^G_{(a, \phi)}.$ We denote by $K^G_{[a,\phi]}$ the Kawanaka character $K^G_{(a,\phi)}$ for each orbit $[a,\phi]$ in~$\mathcal{M}(A)$. \\
 Lastly, we see that if $A$ is an $\ell'$-group, then there exist projective $\k G$-modules $\mathcal{K}_{[a,\phi]}^G$ such that $K^G_{[a, \phi]}$ is the character associated to $(\mathcal{K}^G_{[a, \phi]})^\Oo \otimes_{\Oo} \K$ for any $[a,\phi] \in \Mm(A)$. In particular, if $\ell$ is good for $\G$, then the Kawanaka characters allow us to overcome the difficulty raised in Remark \ref{rmk_plan2}. However, we still have to check the condition (A) stated at the beginning of this section. The idea of Brunat--Dudas--Taylor is to use characteristic functions of character sheaves instead of irreducible characters. 

\subsection{Decomposition of the Kawanaka characters}
\subsubsection{Character sheaves and characteristic functions}
We very briefly recall some results about characteristic functions of character sheaves, for more details see Section \ref{CharSheaves}. \\
Character sheaves are certain irreducible perverse sheaves on $\G$. We denote the set of isomorphism classes of character sheaves of $\G$ by $\hat{\G}$. 
We have \[\hat{\G} = \bigsqcup \hat{\G}_s,\]
where $s$ runs over representatives of the $W$-orbits on $\T^*$ (see \cite[Def.~2.10]{lusztigCharacterSheaves1985}, \cite[Cor.~11.4]{lusztigCharacterSheavesIII1985}, for an isomorphism between the Kummer local systems on $\T$ and $\T^*$). We say that a character sheaf is \textbf{unipotent} if it belongs to $\hat{\G}_1$.\\  

We assume that $\G$ is simple and adjoint. 
The parametrisation of $\hat{\G}$ is very similar to the one of $\irr(G)$. Fix $s$ a representative of a $W$-orbit on $\T^*$. To each family $\Ff$ of $\irr(W_{C_\G^*(s)})$ one can associate a family of character sheaves $(\hat{\G}_s)_\Ff$ parametrised by a variant of $\Mm(\mathcal{G}_\Ff)$ for a certain finite group $\mathcal{G}_\Ff$. In the unipotent case, we have $\mathcal{G}_\Ff = \Omega_{\G,\Ff}$ (\cite[17.8.3]{lusztigCharacterSheavesIV1986}).\\

To each family $\mathcal{F}$ of $\hat{\G}$, we can associate a unipotent conjugacy class $\Cc_\mathcal{F}$ of $\G$, called \textbf{unipotent support}. It is the unique unipotent class of $\G$ satisfying the following properties: 
\begin{enumerate}
	\item there exists a conjugacy class $\mathcal{D}$ of $\G$ and a character sheaf $\mathcal{A} \in \mathcal{F}$ such that the unipotent part of $\mathcal{D}$ is $\mathcal{C}$ and $\mathcal{A}_{\vert_{\mathcal{D}}} \neq 0$,
	\item and for any conjugacy class $\mathcal{D}$ of $\G$ such that the unipotent part of $\mathcal{D}$ is different from $\mathcal{C}$ and  whichhas dimension bigger than or equal to the one of $\mathcal{C}$ and for any character sheaf $\mathcal{A} \in \mathcal{F}$, we have  $\mathcal{A}_{\vert_{\mathcal{D}}} = 0$.
\end{enumerate}

Now, we want to consider the character sheaves $\mathcal{A}$ which are $F$-stable. In that case, we fix an isomorphism $F^*\mathcal{A} \cong \mathcal{A}$ as in \cite[25.1]{lusztigCharacterSheaves1986}  and we can define a class function $\chi_\mathcal{A}$, called the \textbf{characteristic function} of $\mathcal{A}$. It depends on the choice of the isomorphism up to a scalar multiple of norm $1$). Then by \cite[$\S$ 25]{lusztigCharacterSheaves1986}, the set of characteristic functions of $F$-stable character sheaves forms an orthornomal basis of the space of class functions of $\G^F$.  \\ 

This time, by \cite[$\S$ 5]{shojiCharacterSheavesAlmost1995}, the $F$-stable characters sheaves in $\hat{G}_s$ are parametrised by the $F^*$-stable families $\Ff$ of $\irr(W_{C_\G^*(s)})$ for $s$ a representative of a $W$-orbit on $\T^*$ such that $F^*(s) = w.s$ for some $w \in W$. 

\begin{thm}[{\cite{shojiCharacterSheavesAlmost1995},\cite{shojiCharacterSheavesAlmost1995a}}]\label{thm_baseCS}
	The space spanned by the unipotent characters with wave front set $\Cc$ is equal to the space spanned by the Alvis--Curtis duals of the characteristic functions of the $F$-stable unipotent character sheaves with unipotent support $\Cc$.
\end{thm}

\subsubsection{Fourier transform of Kawanaka characters and some of their properties}
We collect some results about Kawanaka characters, coming from \cite[Section 8]{brunatUnitriangularShapeDecomposition2020}.  We fix  an $F$-stable unipotent conjugacy class $\Cc \in X_\G$, $u \in \Cc^F$, an admissible pair $(A, \lambda)$ for $u$ and a set $\{u_a \mid a\in A\}$ of admissible representatives.

%


\begin{dfn}
The \textbf{Fourier transform} of Kawanaka characters is given as follows. For $[a, \phi] \in \Mm(A)$, we set \[F^G_{[a,\phi]} \defeq \sum_{[b,\psi] \in \Mm(A)} \{[a,\phi],[b,\psi]\}K^G_{[b,\psi]}.\] 
\end{dfn}

%
%

%
%
%
%

Using the previous results and knowledge about the characteristic function of some character sheaves, one can get interesting results. We explain one simplified proof of \cite[Thm 8.9]{brunatUnitriangularShapeDecomposition2020} in a particular case. 
Let $\Cc \in X_\G$ and $g \in \G^*$ such that $\Phi(g) = \Cc$. We write $\hat{G}_g $ for the set of  $F$-stable characters sheaves in $\hat{\G}_{g_s}$ with unipotent support $\Cc$. For $\theta$ a class function of $G$, we denote by $\pr_g(\theta)$ the projection of $\theta$ to the space spanned by $D_\G(\irr(G)_g) \defeq \{D_\G(\rho) \mid \rho \in \irr(G)_g\}$. By \cite{shojiCharacterSheavesAlmost1995a}, observe that the space spanned by  $D_\G(\irr(G)_g)$ equal the space spanned by $\{ D_\G(\chi_\mathcal{A}) \mid \mathcal{A} \in \hat{G}_g\}.$

 
\begin{prop} \label{prop_DecKawExc} Assume that $\G$ is simple adjoint of exceptional type. Assume that $\ell$ is good for $\G$. Let  $g \in \G^*$ such that $\Phi(g) = \Cc$. We write $s \defeq g_s$, $v \defeq g_u$ and $\G_s^* \defeq C_{\G^*}(s)$. Assume the following: 
	\begin{itemize}
		\item $F^*$ acts trivially on $\Omega_{\G_s^*,(v)_{\G_s^*}}$,
		\item there exists an $F$-stable unipotent element $u_\Cc \in \Cc$ such that $F$ acts trivially on $A_\G(u_\Cc)$ and an admissible pair $(A_\Cc,\lambda)$ for $u_\Cc$ which is an admissible covering of $A_\G(u_\Cc)$,
		\item $A_\Cc \cong A_\G(u_\Cc) \cong \Omega_{\G_s^*,(v)_{\G_s^*}}$,
		\item there exists an $F$-stable character sheaf $\mathcal{A} \in \hat{G}_g$ such that for all $[b,\phi] \in \Mm(A_\Cc)$ \[\langle F^G_{[b,\phi]}, D_\G(\chi_\mathcal{A}) \rangle = \begin{cases}
			x_\mathcal{A}  &\text{ if } [b,\phi] =[1,1] \\
			0 &\text{ otherwise,}
		\end{cases}\]
		for some $x_\mathcal{A} \in \C^\times$.
	\end{itemize}
	Then $\pr_g(K^G_{[a,\phi]})$ is irreducible for all $[a, \phi] \in \Mm(A_\Cc)$. Furthermore, \[\{\pr_g(K^G_{[a,\phi]}) \mid [a,\phi] \in \Mm(A_\Cc)\} = D_\G(\irr(G)_g)\]
	 and $K^G_{[a,\phi]} = K^G_{[b,\psi]}$ if and only if $[a,\phi] = [b,\psi]$, for $[a,\phi], [b,\psi] \in \Mm(A_\Cc)$.
\end{prop}
\begin{proof} We let $E \defeq \{ D_\G(\chi_\mathcal{A}) \mid \mathcal{A} \in \hat{G}_g\}.$ We fix $[a,\phi] \in \Mm(A_\Cc)$. Then, we have \begin{align*}
		K^G_{[a,\phi]} &= \sum_{[b,\psi] \in \Mm(A_\Cc)} \{[a,\phi],[b,\psi]\}F^G_{[b,\psi]}\\
		&=  \{[a,\phi],[1,1]\}F^G_{[1,1]} + \sum_{[b,\psi] \in \Mm(A_\Cc)\backslash{[1,1]}} \{[a,\phi],[b,\psi]\}F^G_{[b,\psi]}\\
		&= \frac{1}{|C_A(a)|}F^G_{[1,1]} + \sum_{[b,\psi] \in \Mm(A_\Cc)\backslash\{[1,1]\}} \{[a,\phi],[b,\psi]\}F^G_{[b,\psi]}.
	\end{align*}
	Thus \[\pr_g({K}^G_{[a,\phi]}) =  \frac{x_\mathcal{A}}{|C_{A_\Cc}(a)|}D_\G(\chi_\mathcal{A}) + \sum_{\mathcal{A}' \in \, \hat{G}_g\backslash\{A\}}x_{\mathcal{A}'}^{[a,\phi]}D_\G(\chi_{\mathcal{A}'}), \]
	with $x_{\mathcal{A}'}^{[a,\phi]} \in \C$ for all $\mathcal{A}' \in \hat{G}_g\backslash\{A\}$.\\
	The set $E$ forms an orthonormal family. Therefore, we get \begin{align*}
		\langle \pr_g({K}^G_{[a,\phi]}) , \pr_g({K}^G_{[a,\phi]})\rangle 
		&= \frac{|x_\mathcal{A}|^2}{|C_{A_\Cc}(a)|^2} + \sum_{\mathcal{A}' \in \, \hat{G}_g\backslash\{\mathcal{A}\}}x_{\mathcal{A}'}^{[a,\phi]}\overline{x_{\mathcal{A}'}^{[a,\phi]}} > 0. 
	\end{align*}
	Now by construction, $\pr_g({K}^G_{[a,\phi]})$ is a character of $G$, thus for all $[a,\phi], [b,\psi]~\in~\Mm(A_\Cc)$, we have \[\langle \pr_g({K}^G_{[a,\phi]} ), \pr_g({K}^G_{[a,\phi]})\rangle  \geq 1 \text{ and } \langle \pr_g({K}^G_{[a,\phi]}) , \pr_g({K}^G_{[b,\psi]})\rangle  \geq 0.\]

	By the decomposition of GGGCs into Kawanaka characters (Eq. (\ref{KawGGGr})), we get for all $a \in A_\Cc$, 
	\[\langle \pr_g({\Gamma}_{u_a}), \pr_g({\Gamma}_{u_a})\rangle \geq \sum_{\phi \in \irr(C_{A_\Cc}(a))}\phi(1)^2\langle \pr_g({K}^G_{[a,\phi]}) , \pr_g({K}^G_{[a,\phi]})\rangle \geq \sum_{\phi \in \irr(C_{A_\Cc}(a))}\phi(1)^2.\]
	On the other hand, since $A_\G(u_\Cc) \cong \Omega_{\G_s^*,(v)_{\G_s^*}}$, we can apply \cite[Rmk.~ 4.4]{geckUnipotentSupportCharacter2008a} to get 
	\[\langle \pr_g({\Gamma}_{u_a}), \pr_g({\Gamma}_{u_a})\rangle = \sum_{\phi \in \irr(C_{A_\Cc}(a))}\phi(1)^2.\]
	
	Consequently,  for all $[a,\phi] \neq [b,\psi]$ in $ \Mm(A_\Cc)$,  \[\langle \pr_g({K}^G_{[a,\phi]}) , \pr_g({K}^G_{[a,\phi]})\rangle  = 1, \text{ and } \langle \pr_g({K}^G_{[a,\phi]}) , \pr_g({K}^G_{[b,\psi]})\rangle  = 0.\]
	In particular, $\pr_g(K^G_{[a,\phi]})$ is irreducible for all $[a, \phi] \in \Mm(A_\Cc)$.
	Since by assumption,  $F$ acts trivially on $\Omega_{\G_s^*,(v)_{\G_s^*}}$, we have that $|\Mm(\Omega_{\G_s^*,(v)_{\G_s^*}})| = |\irr(G)_g| = |D_\G(\irr(G)_g)|$ and we can conclude. 
\end{proof}  

In the unipotent case, we know more about the characteristic functions of unipotent character sheaves thanks to \cite[Thm.~2.4]{lusztigRestrictionCharacterSheaf2015}. Using this, Brunat--Dudas--Taylor showed a similar result to the above proposition in more generality. 
\begin{prop}[{\cite[Thm.~8.9]{brunatUnitriangularShapeDecomposition2020}}]\label{projkaw}Assume that $\ell$ is good for $\G$. Assume that $\Cc \in X_\G$ is special and $F$-stable. Let $A_\Cc$ be as in Proposition \ref{prop_admcov}.
Given $[a,\phi] \in \Mm(A_\Cc)$, the character $K^G_{[a,\phi]}$ has at most one unipotent constituent with wave front set $\Cc$ and it occurs with multiplicity one. Furthermore, every unipotent character with wave front set $\Cc$ occurs in some $K^G_{[a,\phi]}$ for some $[a,\phi]\in \Mm(A)$.
Moreover, if $A \cong \Omega_{\G,\Cc}$, then $K^G_{[a,\phi]} = K^G_{[b,\psi]}$ if and only if $[a,\phi] = [b,\psi]$, for $[a,\phi], [b,\psi] \in \Mm(A)$.
\end{prop}

\begin{rmk}
We observe that \cite[Thm.~2.4]{lusztigRestrictionCharacterSheaf2015} does not hold in full generality. For instance, if we consider $\G$ simple adjoint of type $E_7$, there are two cuspidal unipotent character sheaves (\cite[Prop.~20.3 c]{lusztigCharacterSheavesIV1986}). Their support is the closure of the $\G$-conjugacy class $(su)_\G$ where $s \in \G$ is semisimple with connected centralizer of type $\SL_4 \times \SL_4 \times \SL_2$ and $u \in C^\circ_\G(s)$ is unipotent such that $(u)_{C^\circ_\G(s)}$ is the regular class. Their associated local systems correspond to the two non-real characters of $A_\G(su) \cong C_4$. They belong to the same family of exceptional character sheaves with unipotent support $(u)_\G$ denoted $A_4 + A_1$ in CHEVIE notation. Theorem 2.4 in \cite{lusztigRestrictionCharacterSheaf2015} claims that the restriction of those character sheaves to their support is a local system corresponding to the lift of a character of $A_\G(u) \cong S_2$, which is necessarily real. \\
Similar situations occur for $\G$ of type $E_8$, when considering cuspidal characters in an exceptional family. However, by explicit computations in CHEVIE \cite{michelDevelopmentVersionCHEVIE2015}, we have checked that the theorem \cite[Thm.~2.4]{lusztigRestrictionCharacterSheaf2015} holds true in exceptional type groups for all the other cases. \\
This discussion does not really matter for the proof of the result of Brunat--Dudas--Taylor. For instance, we could apply Proposition \ref{prop_DecKawExc} in the cases for which \cite[Thm.~2.4]{lusztigRestrictionCharacterSheaf2015} fails to hold. 
\end{rmk}


\section{The unipotent blocks for $\ell$ bad}
\begin{hyps}
	In this section, we assume that $\G$ is simple adjoint of exceptional type and $p$ is good for $\G$.
\end{hyps}
We want to show the unitriangularity of the $\ell$-decomposition matrix of the unipotent $\ell$-blocks for $\ell$ bad. We first use an approach from \cite{geckUnipotentSupportCharacter2008a}, which uses only GGGC's. Then, we will see how to adapt \cite[Thm.~8.9]{brunatUnitriangularShapeDecomposition2020} to the special classes when $\ell$ is bad, adapting the definition of Kawanaka characters. Lastly, we will finish the proof using a case by case analysis. 

\subsection{Using generalised Gelfand--Graev characters}
We study $\ell$-special unipotent conjugacy classes. If we add some conditions on these classes, then it will be enough to choose generalised Gelfand--Graev characters, instead of Kawanaka characters. Recall that a unipotent conjugacy class is said $\ell$-special if its preimage under Lustig's map $\Phi$ contains an element $g$ whose semisimple part is isolated and an $\ell$-element. We could ask for another property of $g$.

\begin{dfn}[{\cite[Thm.~A]{hezardSupportUnipotentFaisceauxcaracteres2004}}]\label{propertyP}
	Let $\Cc$ be an $F$-stable unipotent class of $\G$ and $g$ be a special element of $\G^*$. We say that $g$ satisfies the \textbf{property $(P)$} with respect to $\Cc$ if :
	\begin{enumerate}
		\item $\Phi((g)_{\textbf{G}^*}) = \Cc,$
		\item $|\Omega_{C_{\textbf{G}^*}(g_s), g_u}|=|A_\textbf{G}(u_\Cc)|$, and
		\item the image of $g_s$ in the adjoint quotient of $\textbf{G}^*$ is quasi-isolated.
	\end{enumerate}
	If there exists such a $g \in \G^*$ such that moreover, $g_s$ is isolated and an $\ell$-element, we say that $\Cc$ is \textbf{$\ell$-$P$-special}.
\end{dfn}

In general, for any $F$-stable unipotent conjugacy class of $\G$, there exists a special element $g \in \G^*$ satisfying $(P)$ with respect to $\Cc$ and such that the class $(g)_{\G^*}$ is $F^*$-stable (\cite[Thm.~B]{hezardSupportUnipotentFaisceauxcaracteres2004}). By the Lang--Steinberg theorem, we may choose $g \in (\G^*)^{F^*}$. However, not all $\ell$-special classes are $\ell$-$P$-special.  

\begin{lem}\label{lem_proplPspecial} 
	Let $\Cc$ be an $F$-stable unipotent class of $\G$. Assume that $\ell$ is bad for $\G$. Then 
	\begin{itemize}
		\item $\Cc$ is $\ell$-$P$-special if and only if $\Omega^\ell_{\G,\Cc} = A_\G(u_\Cc)$,
		\item if $\Cc$ is $\ell$-special but not special then $\Cc$ is $\ell$-$P$-special, and 
		\item if $\Cc$ is $\ell$-special but not $\ell$-$P$-special, then $\Omega^\ell_{\G,\Cc} \cong \Omega_{\G,\Cc} \neq A_\G(u_\Cc)$.
	\end{itemize}
\end{lem}
\begin{proof}
	We get these results by computations using CHEVIE \cite{michelDevelopmentVersionCHEVIE2015}, which has all the data for the $j$-induction, the Springer correspondence and the isolated semisimple elements.  
\end{proof}
 
%
%
%
%
%
%
%

For $\Cc \in X_\G$ an $F$-stable class, we know part of the restriction of the GGGC's coming from $\Cc$ to irreducible characters of $G$ with unipotent support $\Cc$. 
\begin{prop}[{\cite[Prop.~3.4]{geckUnipotentSupportCharacter2008a}}]\label{prop_avtgenial}Assume that $p$ is good for $\G$.
	Let $\Cc$ be an $F$-stable unipotent class of $\G$ and $u_1,\dots,u_d$ be representatives for the $G$-conjugacy classes contained in $\Cc^F$. Let $g \in (\G^*)^{F^*}$ satisfying property $(P)$ with respect to $\Cc$. Assume that $\Omega_{C_{\textbf{G}^*}(g_s), g_u}$ is abelian. Then there exist $\rho_1,\dots,\rho_d \in \irr(G)_g$ such that $\langle\rho^*_i,\gamma_{u_j} \rangle = \delta_{ij}$ for $1 \leq i,j\leq d$.
\end{prop}

As a corollary, if $\Cc$ as above is $\ell$-$P$-special and $d = \alpha^\ell_{\Cc}$, then considering the generalised Gelfand-Graev characters is sufficient.

\begin{cor} \label{cor_genial} 
	Let $\Cc$ be an $F$-stable $\ell$-$P$-special unipotent class of $\G$. If $\Omega^\ell_{\G,\Cc}$ is trivial or $\ell = 2$ and $\Omega^\ell_{\G,\Cc} \cong S_2$, then there exist $\rho_1,\dots,\rho_{\alpha^\ell_{\Cc}} \in \irr(G)$ in the unipotent $\ell$-blocks with unipotent support $\Cc$ and generalised Gelfand--Graev characters $\Gamma_1, ..., \Gamma_{\alpha^\ell_{\Cc}}$  such that for $1 \leq i,j\leq \alpha^\ell_{\Cc},$
	\[\langle\rho^*_i,\gamma_{j} \rangle = \delta_{ij}.\]
\end{cor}

\begin{prop}\label{prop_casesleft}
If $\G$ is simple adjoint of exceptional type and $\ell$ is bad for $\G$, the only $\ell$-special but not special unipotent conjugacy classes of $\G$ for which we can not apply Corollary \ref{cor_genial} are when $\G$ is of type $E_8$
\begin{enumerate}
	\item  $\ell = 2$ and the unipotent conjugacy class is $E_7(a_5)$ and
	\item $\ell = 3$ and the unipotent conjugacy class is $E_6(a_3) + A_1$.
\end{enumerate}
\end{prop}
\begin{proof}
	This follows by inspection of the tables in Appendix \ref{Tables}.
\end{proof}


\subsection{Adapting the general approach for the special classes}

In the leftover cases, we would like to adapt the proof of \cite[Thm.~8.9]{brunatUnitriangularShapeDecomposition2020}.  We fix $\Cc$ an $F$-stable unipotent conjugacy class, $u \in \Cc^F$, an admissible pair $(A, \lambda)$ for $u$ and a set $\{u_a \mid a\in A\}$ of admissible representatives.\\
The main issue is that Kawanaka characters might not be characters of projective $\Oo G$-modules anymore, because the group $A$ might not be an $\ell'$-group. 

\subsubsection{$\ell$-Kawanaka characters}
\begin{dfn}
	Assume $a \in A$ and let $\Psi$ be the character of a projective indecomposable $\k C_A(a)$-module $P$ (i.e. the character associated to the $\K C_A(a)$-module $P^\Oo \otimes_\Oo \K$).  We also write $W_{u_a}$ for a module of $\K C_A(a) \ltimes \U_\lambda(-1)^F$ affording the character $\tilde{\xi}^G_{u_a}$ for any $a \in A$.
	We define the \textbf{$\ell$-Kawanaka module} associated to the pair $(a,\Psi)$ to be
	  \[\mathcal{K}^G_{(a, \Psi)} \defeq \Ind^{G}_{C_A(a) \ltimes \U_\lambda(-1)^F}\left( \left( (W_{u_a})_\Oo \otimes_\Oo \k \right)  \otimes \Inf^{C_A(a) \ltimes \U_\lambda(-1)^F}_{C_A(a)}P\right) \]
	and denote by $K^G_{(a,\Psi)}$ the character afforded by the module $(\mathcal{K}^G_{(a,\Psi)})^\Oo \otimes_\Oo \K.$  
\end{dfn}

\begin{rmk}
	Observe that if $C_A(a)$ is an $\ell'$-group, $\Psi$ is an irreducible character of $C_A(a)$. 
\end{rmk}

\begin{lem}
	Let $a \in A$ and let $\Psi$ be the character of a projective indecomposable $\k [C_A(a)\ltimes \U_\lambda(-1)^F]$-module $P$. Then $\mathcal{K}^G_{(a,\Psi)}$ is a projective $\k G$-module. 
\end{lem}
\begin{proof}

Since $\U_\lambda(-1)^F$ is a $p$-group, and $p \neq \ell$, the inflation of $P$ is a projective $\k C_A(a)$-module. Tensoring and inducing preserve projectivity, thus $\mathcal{K}^G_{(a,\Psi)}$ is projective.  
\end{proof}

\begin{lem}\label{Kaw_char}
Fix $a \in A$ and let $\Psi$ be the character of a projective indecomposable $\k C_A(a)$-module $P$. Then  \[ K^G_{(a,\Psi)} = \sum_{\psi \in \irr(C_A(a))} d_{\psi, \Psi}K^G_{(a,\psi)}.\]

\end{lem}
\begin{proof} 
	For any $\psi \in \irr(C_A(a))$ we write $V_\psi$ for an irreducible $\K C_A(a)$-module affording the character $\psi$.
	We observe that $K^G_{(a,\Psi)}$ is the character of 
	\begin{align*} &  \Ind^{G}_{C_A(a) \ltimes \U_\lambda(-1)^F}\left( \left( (W_{u_a})_\Oo \otimes \k \right)  \otimes \Inf^{C_A(a) \ltimes \U_\lambda(-1)^F}_{C_A(a)}P\right)^\Oo \otimes_\Oo \K, \\
		&=\Ind^G_{C_A(a) \ltimes \U_\lambda(-1)^F}\left(\left((W_{u_a})_\Oo \otimes \k \otimes \Inf^{C_A(a) \ltimes \U_\lambda(-1)^F}_{C_A(a)}P\right)^\Oo \otimes_\Oo \K\right),\\
		&=\Ind^G_{C_A(a) \ltimes \U_\lambda(-1)^F}\left(W_{u_a} \otimes \left(\Inf^{C_A(a) \ltimes \U_\lambda(-1)^F}_{C_A(a)}P\right)^\Oo \otimes_\Oo \K\right),\\
		&= \Ind^G_{C_A(a) \ltimes \U_\lambda(-1)^F}\left(W_{u_a} \otimes \Inf^{C_A(a) \ltimes \U_\lambda(-1)^F}_{C_A(a)}P^\Oo \otimes_\Oo \K\right),\\
		&= \Ind^G_{C_A(a) \ltimes \U_\lambda(-1)^F}\left(W_{u_a} \otimes \Inf^{C_A(a) \ltimes \U_\lambda(-1)^F}_{C_A(a)}\sum_{\psi \in \irr(C_A(a))} d_{\psi, \Psi}V_\psi\right),\\
		& = \sum_{\psi \in \irr(C_A(a))} d_{\psi, \Psi}\Ind^G_{C_A(a) \ltimes \U_\lambda(-1)^F}\left(W_{u_a} \otimes \Inf^{C_A(a) \ltimes \U_\lambda(-1)^F}_{C_A(a)}V_\psi\right).
\end{align*}

Thus, $K^G_{(a,\Psi)} = \sum_{\psi \in \irr(C_A(a))} d_{\psi, \Psi}K^G_{(a,\psi)}$.
\end{proof}

As a consequence, since for any $a,b \in A$ and $\phi \in \irr(C_A(a))$, we have $K^G_{({}^ba,{}^b\phi)} = K^G_{(a, \phi)},$ it makes sense to write $K^G_{[a,\Psi]} \defeq K^G_{(a,\Psi)}$ for any $[a, \Psi] \in \Mm^\ell(A).$
%
%
%

\subsubsection{Special unipotent conjugacy classes}

For the special conjugacy classes, we want to use what we know about the decomposition of Kawanaka characters into irreducible characters of $G$ to deduce the decomposition of the $\ell$-Kawanaka characters. 

\begin{prop}\label{prop_specialexcellent} Assume that $\G$ is adjoint simple of exceptional type. Let $\Cc$ be a special $F$-stable unipotent conjugacy class of $\G$ and $(A, \lambda)$ be the admissible covering of the ordinary canonical quotient of $A_\G(u_\Cc)$ as in Proposition \ref{admcovexcp}. Let $d = |\mathcal{M}^\ell(\Omega_{\G,\Cc})|.$
Assume that either $\ell$ does not divide $|A|$ or $A \cong \Omega_{\G,\Cc}$.
Then, there exist unipotent characters $\rho_1,\dots,\rho_d$ of $G$ with unipotent support $\Cc$ and $[a_1,\Psi_1], \dots, [a_d,\Psi_d] \in \mathcal{M}^\ell(A)$ such that  for $1 \leq i,j\leq d$, \[\langle \rho^*_i, K^G_{[a_j,\Psi_j]} \rangle = \begin{cases}
	0 \quad i<j,\\
	1 \quad i =j.\\
	\end{cases}\] 

\end{prop}

\begin{proof} Let $v \in (\G^*)^{F^*}$ be unipotent such that $\Phi(v) = \Cc$.  
We observe that, by Lemma \ref{Kaw_char}, \[\pr_v({K}^G_{(a,\Psi)}) = \sum_{\psi \in \irr(C_A(a))} d_{\psi,\Psi}\pr_v({K}^G_{(a,\psi)}).\]
	
	Let us first assume that $\ell$ does not divide $|A|$. Since $\Mm^\ell(\Omega_{\G,\Cc})=\Mm(\Omega_{\G,\Cc})$, which is in bijection with the set of unipotent characters with unipotent support~$\Cc$, this is just a reformulation of Proposition \ref{projkaw}.
	


	Suppose now that $A \cong \Omega_{\G,\Cc}$. Thanks to Proposition \ref{projkaw}, up to reindexing, we can write $\pr_v({K}_{[a,\psi]}) = \rho^*_{[a,\psi]}$, where $ \rho_{[a,\psi]}$ is a unipotent character with unipotent support $\Cc$ for all $[a,\psi] \in \Mm(\Omega_{\G,\Cc}).$ \\
	In other words, for each $a \in A$ and $ \Psi$ the character of a projective indecomposable $\k C_A(a)$-module, we have 
	\[\pr_v({K}^G_{(a,\Psi)}) = \sum_{\psi \in \irr(C_A(a))} d_{\psi,\Psi}\pr_v({K}^G_{(a,\psi)}) = \sum_{\psi \in \irr(C_A(a))} d_{\psi,\Psi}\rho^*_{[a,\psi]}.\]
	Now, we observe that for $\psi,\psi' \in \irr(C_A(a))$, we have $[a,\psi]= [a,\psi']$ if and only if $\psi = \psi'$. Therefore, \[\langle \rho^*_{[a,\psi]},  \pr_v({K}^G_{[a,\Psi]})\rangle  = d_{\psi,\Psi}.\]
	On the other hand, for $b \in A$ not $A$-conjugate to $a$, and any $\phi \in \irr(C_A(b))$, we have $\langle \rho^*_{[b,\phi]},  \pr_v({K}^G_{[a,\Psi]})\rangle = 0.$
	Assume that for each $a \in A$, we have fixed a total ordering of  $\{ \Psi_j \mid 1 \leq j \leq~s_a\}$, the set of characters of $C_A(a)$ associated to the projective indecomposable $\k C_A(a)$-modules, and an ordering of $\{\psi_i \mid 1 \leq i \leq t_a\} = \irr(C_A(a))$ such that for all $1 \leq j \leq s_a$ and for $1 \leq i \leq j$,\[ d_{\psi_i,\Psi_j} =  \begin{cases}
		0 \text{ if } i <j\\
		1 \text{ if } i = j.\end{cases}\]
	Then for each $1 \leq j \leq s_a$, we set $\rho_{[a,\Psi_j]} \defeq \rho_{[a,\psi_j]}$ and the sets $\{\rho_{[a, \Psi]} \mid [a, \Psi] \in\Mm^\ell(A)\}$ and $\{K^G_{[a, \Psi]} \mid [a, \Psi] \in\Mm^\ell(A)\}$ satisfy the statement of the proposition.

In other terms, we are left to check that the $\ell$-decomposition matrix of $C_A(a)$ is lower-unitriangular for each $a \in A$. One can easily check that this holds since by Proposition \ref{admcovexcp}, the group $A$ is either $S_2, S_3, S_4$ or $S_5$ and the primes are $\ell \in\{2,3,5\}$. We need to check the $\ell$-decomposition matrices of the following groups: $S_2, S_3, S_4, S_5$, $C_3, C_2 \times C_2, D_8, C_4, C_5$ and $D_{12}$ (group with $12$ elements) and $ C_6$. We already know that the $\ell$-decomposition matrix of the symmetric group is unitriangular. Moreover, it is also trivially the case for groups of order a prime power. We can easily check that it is also true for the last two cases. 
\end{proof}
%
\begin{cor}\label{cor_excep}
Assume that $\G$ is adjoint simple of exceptional type. Let $\mathcal{C}$ be a special $F$-stable unipotent conjugacy class of $\G$ and $(A, \lambda)$ be the admissible covering of the ordinary canonical quotient of $A_\G(u_\Cc)$, as in Proposition \ref{admcovexcp}. Assume that $\Omega_{\G,\Cc}  \cong \Omega^\ell_{\G,\Cc}$ and that either $\ell$ does not divide $|A|$ or $A \cong \Omega_{\G,\Cc}$.
Then, there exist $\rho_1,\dots,\rho_{\alpha^\ell_{\Cc}}$ unipotent characters of $G$ with unipotent support $\Cc$ and $[a_1,\Psi_1], \dots, [a_{\alpha^\ell_{\Cc}},\Psi_{\alpha^\ell_{\Cc}}]\in\mathcal{M}^\ell(A)$ such that  for $1 \leq i,j\leq {\alpha^\ell_\Cc}$, \[\langle \rho^*_i, K_{[a,\Psi]_j} \rangle = \begin{cases}
	0 \quad i<j,\\
	1 \quad i =j.\\
\end{cases}\] 
\end{cor}

Observe that the proof of Proposition \ref{prop_specialexcellent} also shows the following result:
\begin{prop} \label{prop_excellent}
	Assume that $\G$ is adjoint simple of exceptional type. Let $\Cc$ be an $F$-stable unipotent conjugacy class of $\G$. Assume that \begin{enumerate}
			\item there exists an $F$-stable unipotent element $u_\Cc \in \Cc$ such that $F$ acts trivially on $A_\G(u_\Cc)$ and an admissible pair $(A_\Cc,\lambda)$ for $u_\Cc$ which is an admissible covering of $A_\G(u_\Cc)$,
			\item there is $g \in (\G^*)^{F^*}$ such that $\Phi(g) = \Cc$ and $g_s$ is an $\ell$-element
			\item $\Omega_{\G_s^*,(v)_{\G_s^*}} \cong \Omega^\ell_{\G,\Cc}$,
			\item there is $g \in (\G^*)^{F^*}$ such that $\Phi(g) = \Cc$ and $g_s$ is an $\ell$-element, and
			\item $\{\pr_g(K^G_{[a,\phi]}) \mid [a,\phi] \in \Mm(A_\Cc)\} = D_\G(\irr(G)_g)$
			and if $A_\Cc \cong \Omega_{\G_s^*,(v)_{\G_s^*}}$ and $F^*$ acts trivially on $ \Omega_{\G_s^*,(v)_{\G_s^*}}$, then $K^G_{[a,\phi]} = K^G_{[b,\psi]}$ if and only if $[a,\phi] = [b,\psi]$, for $[a,\phi], [b,\psi] \in \Mm(A_\Cc)$.
	\end{enumerate}
	If either $\ell$ does not divide $|A_\Cc|$ or  $A_\Cc \cong  \Omega_{\G_s^*,(v)_{\G_s^*}}$, then there exist $\rho_1,\dots,\rho_{\alpha^\ell_{\Cc}} \in \irr(G)_g$ in the unipotent blocks with unipotent support $\Cc$ and $[a_1,\Psi_1], \dots, [a_{\alpha^\ell_{\Cc}},\Psi_{\alpha^\ell_{\Cc}}]\in\mathcal{M}^\ell(A_\Cc)$ such that  for $1 \leq i,j\leq {\alpha^\ell_\Cc}$, \[\langle \rho^*_i, K_{[a,\Psi]_j} \rangle = \begin{cases}
		0 \quad i<j,\\
		1 \quad i =j.\\
	\end{cases}\]   
\end{prop}

\begin{lem}
If $\G$ is simple adjoint of exceptional type, the only special unipotent classes of $\G$ for which we can not apply Corollary \ref{cor_excep} are given in Table \ref{tablepasexcellent}.
\begin{table}[h!]
	\centering
	$\begin{array}{|p{0.05\textwidth}|p{0.5\textwidth}|p{0.075\textwidth}|}
		\hline
		$\G$ & $\ell =2$ & $\ell = 3$ \\ \hline
		$F_4$ & $A_2,$ $ F_4(a_2)$ & \\ \hline
		$E_7$ & $ A_4 + A_1,$ $ E_7(a_4),$ $A_3 + A_2$ & \\ \hline
		$E_8$ & $E_6(a_1) +A_1,$ $D_7(a_2),$ $A_4 + A_1,$ $ E_8(b_4), $ $D_7(a_1),$ $ D_5 + A_2,$ $ E_7(a_4),$ $ D_4+ A_2, $ $A_3 + A_2$ & $E_8(b_6)$ \\\hline
	\end{array}$
	\caption{Special unipotent conjugacy classes where we can not apply Cor. \ref{cor_excep}. }
	\label{tablepasexcellent}
\end{table}
Moreover, the unipotent class $E_8(b_6)$ for $\G$ of type $E_8$ when $\ell=3$ is the only one of those exceptions which is not $\ell$-$P$-special. 
\end{lem}
\begin{proof}
This follows from the description of the admissible covering in Proposition \ref{admcovexcp} and from explicit computations in CHEVIE \cite{michelDevelopmentVersionCHEVIE2015} of the ordinary and $\ell$-canonical quotients, as well as the computation of Lusztig's map, which uses the data in CHEVIE for the $j$-induction, the Springer correspondence and the isolated elements.   
\end{proof}

The only special unipotent conjugacy class where we can not apply Corollary \ref{cor_excep} nor Corollary \ref{cor_genial} is when $\G$ is of type $E_8$, $\ell = 3$, and the unipotent conjugacy class is $E_8(b_6)$. We study this case separately. 

\begin{lem}
	Let $\G$ be a simple group of type $E_8$ and $\Cc$ the $F$-stable unipotent class $E_8(b_6)$. Then, there exist $\alpha_\Cc^3$ irreducible characters in the unipotent $3$-blocks with unipotent support $\Cc$ and $\alpha_\Cc^3$ projective characters (either Kawanaka or GGGC) such that the decomposition matrix restricted to these rows and columns is unitriangular.  
\end{lem} 

\begin{proof}
	We observe using CHEVIE \cite{michelDevelopmentVersionCHEVIE2015} that $A_\G(u_\Cc)\cong \Omega^3_{\G,\Cc} \cong S_3$ and $\alpha_\Cc^3= 5$.\\
	Firstly, thanks to Proposition \ref{admcovexcp}, we can find an admissible covering $A$ of the ordinary canonical quotient associated to $\Cc$. In this case, we have $A \cong \Omega_{\G,\Cc} \cong S_2$. We denote the elements of $\Mm(A)$ by $[1,1], [1,\sgn], [-1,1],[-1,\sgn]$, where $\sgn$ denotes the sign character. 
	Thanks to Proposition \ref{prop_specialexcellent} and since $\ell$ does not divide $A$, we can find four unipotent characters $\rho_{[1,1]},\rho_{[1,\sgn]},\rho_{[-1,1]},\rho_{[-1,\sgn]}$ with unipotent support $\Cc$, and construct four Kawanaka characters with respect to $A$ and $\Cc$ such that for $[b,\phi], [a,\psi] \in \Mm(A)$, \[\langle \rho^*_{[b,\phi]}, K^G_{[a,\psi]} \rangle = \begin{cases} 1 \quad b = a \text{ and } \phi = \psi ,\\
		0 \quad \text{otherwise}.\\
	\end{cases}\] Since $|\mathcal{M}^3(\Omega^3_{\G,\Cc})| = 5$, we need to find an irreducible character of $G$ in the unipotent $\ell$-blocks, which has unipotent support $\Cc$ but is not unipotent. As in \cite[Proof of Prop.~4.3]{geckUnipotentSupportCharacter2008a}, we have for any unipotent character $\rho$ with unipotent support $\Cc$, \begin{align}\label{eq2}
	\sum^{3}_{i =1} [A_\G(u_i) : A_\G(u_i)^F]\langle \rho^*, \gamma^G_{u_i}\rangle = \frac{|A_\G(u)|}{n_\rho},
	\end{align}
	
	\noindent where $n_\rho$ is given by \cite[4.26.3]{lusztigCharactersReductiveGroups1984}. In our case, since $\rho$ is unipotent and $\Omega_{\G,\Cc} \cong S_2$, we have $n_\rho = 2.$ Moreover, as in \cite[Example 2.7.8 c)]{geckCharacterTheoryFinite2020}, we may assume that $u_1 $ corresponds to $1$ (whence $[A_\G(u_1) : A_\G(u_1)^F] = 1$), $u_2$ corresponds to a $2$-cycle (whence $[A_\G(u_2) : A_\G(u_2)^F] = 3$), and $u_3$ to a $3$-cycle (whence $[A_\G(u_3) : A_\G(u_3)^F] =2$). \\
	Let $\phi \in \irr(S_2)$ and $i,j \in \{\pm 1\}$. By Equation (\ref{KawGGGr}), there are two distinct GGGCs, say $\gamma^G_{v_1}$ and $\gamma^G_{v_{-1}}$ such that $\langle \rho^*_{[i,\phi]},\gamma^G_{v_j}\rangle = \delta_{i,j}.$ By construction, we have $\gamma^G_{v_1} = \gamma^G_{u_1}$ and $\gamma^G_{v_{-1}} = \gamma^G_{u_2}$. Applying to (\ref{eq2}), we must have \[\langle \rho^*_{[1,\phi]},\gamma^G_{u_3}\rangle = 1 \text{ and } \langle \rho^*_{[-1,\phi]},\gamma^G_{u_3}\rangle = 0.\]
	
	 Moreover, we can check using CHEVIE \cite{michelDevelopmentVersionCHEVIE2015} that the conjugacy class $\Cc$ is $3$-$P$-special. In other words, there is $g =sv=vs \in (\G^*)^{F^*}$ with $s \in (\G^*)^{F^*}$ semisimple of order a power of $3$, and $v \in (\G^*)^{F^*}$ unipotent such that $\Omega_{C_{\textbf{G}^*}(s), v} \cong S_3$ and  $\Phi(g) = \Cc$. Now by \cite[Prop.~6.7]{geckCharacterSheavesGeneralized1999} and \cite[Rem.~4.4]{hezardSupportUnipotentFaisceauxcaracteres2004}, there is a character $\mu \in \mathcal{E}(\G,s)_g$ such that 
	 \[\langle \mu^*, \gamma^G_{u_i} \rangle = \delta_{3i}.\] We then choose the irreducible characters $\mu,$ $\rho_{[1,1]},$ $\rho_{[1,\sgn]},$ $\rho_{[-1,1]},$ $\rho_{[-1,\sgn]}$ and the projective $\k G$-characters
	 $\gamma^G_{u_3}, \linebreak[2]K^\G_{[1,1]}, K^\G_{[1,\sgn]},K^\G_{[-1,1]},K^\G_{[-1,\sgn]}$ in these orders.
	  We can check that the decomposition matrix of $G$ restricted to these rows of irreducible $\K G$-modules and columns of projective $\k G$-modules has the following shape, where the empty entries are~$0$:
	 $\left( \begin{smallmatrix}
	  	1 & &&&\\
	  	1&1&&&\\
	  	1&&1&&\\
	  	&&&1&\\
	  	&&&&1
	  \end{smallmatrix}\right). $
\end{proof}
Thanks to Proposition \ref{prop_casesleft} and the above discussion on the special classes, we have only two $\ell$-special but not special classes to consider, which occur when $\G$ is of type $E_8$. In order to apply Proposition \ref{prop_excellent}, we need to understand the decomposition of the Kawanaka characters into non-unipotent characters, for instance using Proposition \ref{prop_DecKawExc} and non-unipotent character sheaves. 

\section{Restriction of principal series character sheaves to conjugacy classes} \label{CharSheaves}

In this section, we will consider the restriction of character sheaves coming from the principal series to a mixed conjugacy class.

\begin{rmk}
 All the perverse sheaves are defined over $\Ql$. We do not make assumptions on the center of $\G$ nor on the prime $p$.
\end{rmk}

\subsection{Induction of character sheaves}
From certain local systems on the maximal torus $\T$, we construct  $\G$-equivariant semisimple perverse sheaves on~$\G$, following \cite[$\S$ 1 and 2]{lusztigCharacterSheaves1985}.
\subsubsection{Kummer local systems on the torus and character sheaves}

We consider a certain class of local systems on the torus. Firstly, we fix an injective homomorphism\[\psi: \{x  \in {k}^\times \mid x^n =1 \text{ for some } n \in \N\} \to \Ql^\times.\]  
\begin{dfn}
	We say that a $\Ql^\times$-local system $\Ll$ on $\T$ is \textbf{Kummer} if there is $n \in \N$, coprime to $p$, such that $\Ll^{\otimes n} = \Ql$. 	
\end{dfn}
Kummer local systems are constructed as follows:
\begin{enumerate}
	\item Let $n\in \N$ such that $(p,n) =1$, and $\mu_n \defeq \{x \in {k}^\times \mid x^n = 1\}$. Define $\rho_n: {k} \to {k}$, $x \mapsto x^n$. Then $\mu_n$ acts on the local system $\left(\rho_n\right)_*\Ql$. 
	\item Set $\mathcal{E}_{n,\psi}$ the summand of $\left(\rho_n\right)_*\Ql$ on which $\mu_n$ acts according to $\psi$. 
	\item Fix $\lambda \in \Hom(\T ,{k}^\times)$ and consider the $\Ql^\times$-local systems on $\T $ of the form $\lambda^*\mathcal{E}_{n,\psi}$. 
\end{enumerate} 
We denote by $\mathcal{S}(\T)$ the set of isomorphism classes of Kummer $\Ql^\times$-local systems on $\T$.

The action of $w \in W$ on $\T $ induces an action on $\mathcal{S}(\T )$ sending the isomorphism class of ${\Ll}$ to the isomorphism class of $\ad(\dot{w})^*{\Ll}$. For ${\Ll} = \lambda^*\mathcal{E}_{n,\psi} \in \mathcal{S}(\T )$ with $n \in \N$ and $\lambda \in \Hom(\T,k^\times)$, we define \[W_{{\Ll}} \defeq \{w \in W \mid \ad(w^{-1})^*{\Ll} \cong {\Ll}\}.\]
Observe that $W_{\Ll}$ is not always a Coxeter group. We define \[\Phi_{{\Ll}} \defeq \{\alpha \in \Phi \mid s_a \in W_{\Ll}\},\] and $W^\circ_{\Ll}$ the Weyl group generated by $\{s_\alpha \mid \alpha \in\Phi_{\Ll}\}.$
By \cite[$\S$ 2.2.2]{lusztigCharacterSheaves1985}, for each $w \in W_\Ll$, there exists a character $\lambda_w \in X(\T)$ such that $ \ad(w^{-1})^*{\Ll} = \ad(w^{-1})^*\lambda^*\mathcal{E}_{n,\psi} = (\lambda_w^n\lambda)^*\mathcal{E}_{n,\psi}$. \\

To each $\Ll \in \mathcal{S}(\T)$ corresponds a set of character sheaves $\hat{\G}_\Ll$ as defined by Lusztig in \cite[Def.~2.10]{lusztigCharacterSheaves1985}. We have, by \cite[Prop. 11.2]{lusztigCharacterSheavesIII1985} $$\hat{\G} = \bigsqcup_{{\Ll} \in \mathcal{S}(\T )/W}\hat{\G}_{\Ll}.$$ 

\subsubsection{Induction of character sheaves}
 As for representations of finite groups, we would like to obtain character sheaves of $\G$ from character sheaves of subgroups of $\G$, and in particular from Levi subgroups. For our purpose, we only describe the \textbf{parabolic induction} coming from a torus, following \cite[$\S$ 7.1]{AST_1989__173-174__111_0} and \cite[$\S$3  and 4]{lusztigCharacterSheavesIII1985}.
We write $\G \times_\B \B$ for the quotient of $\G \times \B$ by the $\B$-action $b.(g,q) = (gb^{-1}, bqb^{-1})$ for $b,q \in \B, g \in \G$. We have the following diagram: \begin{center}
	\begin{tikzcd}
		\T & \G \times \B \arrow[l, "\alpha"'] \arrow[r, "\beta"] & \G \times_\B \B \arrow[r, "\delta"] & \G
	\end{tikzcd}
\end{center}
with \begin{itemize}
	\item the map $\alpha: (g,ut) \mapsto t$ for $g \in \G$, $u \in \U$ and $t \in \T$,
	\item the quotient map $\beta$,
	\item and the map $\delta: (g,b) \mapsto gbg^{-1}$ for $g \in \G,  \in \B$.
\end{itemize}
If $K$ is a $\T$-equivariant perverse sheaf on $\T$, then $\alpha^*K[\dim\G+\dim \U]$ is a $\B$-equivariant perverse sheaf on $\G \times \B$. There exists a unique perverse sheaf $\tilde{K}$ (up to isomorphism) on $\G \times_\B \B$ such that $\alpha^*K[2\dim \U] = \beta^*(\tilde{K})$. We put \[\Ind_\B^\G(K) = \delta_*(\tilde{K}).\]

Any irreducible perverse sheaf on a variety $X$ can be obtained as the intersection cohomology complex $IC(\overline{U}, \Ee)$ where $U\subseteq X$ is open and $\Ee$ is a local system on~$U$ \cite[Thm.~4.3.1]{beilinsonAnalyseTopologieEspaces1982}. 
We fix  $$K \defeq IC(\T,\Ll)[\dim \T]$$ where $\Ll =  \lambda^*(\mathcal{E}_{n,\psi}) \in \mathcal{S}(\T)$ is a Kummer local system on $\T$. We construct an intersection cohomology complex isomorphic to the induced perverse sheaf $\Ind_\B^\G(K)$ as follows:

We have the following diagram:\begin{center}
	\begin{tikzcd}
		\T & \G \times \T_{reg}\arrow[l, "\alpha"'] \arrow[r, "\beta"] & \G \times_\T \T_{reg} \arrow[r, "\delta"] & Y_{\T}
	\end{tikzcd}
\end{center}
with \begin{itemize}
	\item the set $\T_{reg} \defeq \{g \in \T \mid C^\circ_\G(g) = \T \}$ and the set $Y_{\T} \defeq \bigcup_{g \in \G}g\T_{reg}g^{-1}$, 
	\item the set $\G \times_\T \T_{reg}$, quotient of $\G \times \T_{reg}$ by the $\T$-action $t.(g,a) = (gt^{-1}, tat^{-1}) = (gt^{-1},a)$ for $t \in \T, g \in \G$ and $a \in \T_{reg}$,
	\item the map $\alpha$ which is the projection on $\T$ of the second coordinate,
	\item the quotient map $\beta$,
	\item and the map $\delta: (g,a) \mapsto gag^{-1}$ for $g \in \G, a \in \T_{reg}$.
\end{itemize}

Note that $\overline{Y_\T} = \G$. 
There exists a unique local system $\tilde{\Ll}$ on $\G \times_\T \T_{reg}$ such that $\alpha^* \Ll = \beta^*(\tilde{\Ll)}$ (up to isomorphism). Then, thanks to \cite[Thm 4.5]{lusztigIntersectionCohomologyComplexes1984} \[\Ind_\B^\G(K) \cong IC(\overline{Y_{\T}},\delta_*(\tilde{\Ll}))[\dim Y_{\T}] = IC(\G,\delta_*(\tilde{\Ll}))[\dim \G].\]
It is semisimple and decomposes into a direct sum of character sheaves. In fact, we can write \[\Ind_\B^\G(K) \cong \bigoplus_{E \in \irr(\End(\Ind_\B^\G(K)))} \Aa_E \otimes V_E,\] where $\Aa_E$ are the irreducible character sheaves in $\Ind_\B^\G(K)$ and $V_E = \Hom(\Aa_E,\Ind_\B^\G(K))$ are the irreducible $\End(\Ind_\B^\G(K))$-modules. This algebra is closely related to some relative Weyl groups.
 \begin{thm}[{\cite[$\S$ 10.2]{lusztigCharacterSheavesII1985}}] The algebra 
	$\End(\Ind_\B^\G(K))$ is isomorphic to the group algebra $\Ql[W_{\Ll}]$ twisted by a $2$-cocycle. 
\end{thm} 
%

\subsubsection{The algebras $\End(\Ind_\B^\G(K))$ and $\Ql[W_\Ll]$}

In the principal series case, Lusztig made this isomorphism more explicit and showed that the cocycle is trivial \cite[$\S$~2.3]{lusztigCharacterValuesFinite1986}. Observe that for any $w \in W$, the following diagram commutes:
\begin{center}
	\begin{tikzcd}
		\T \arrow[d, "\ad(\dot{w}^{-1})"'] & \G \times \T_{reg} \arrow[d, "\varphi_w"] \arrow[r, "\beta"] \arrow[l, "\alpha"'] & \G \times_\T \T_{reg} \arrow[d, "\bar{\varphi}_w"] \arrow[r, "\delta"] & {Y_{\T}} \arrow[d, "id"] \\
		\T                      & \G \times \T_{reg} \arrow[l, "\alpha"] \arrow[r, "\beta"']                        & \G \times_\T \T_{reg} \arrow[r, "\delta"']                             & {Y_{\T}}                
	\end{tikzcd}
\end{center}
with the $\T$-equivariant map $\varphi_w: \G \times \T_{reg} \to \G \times \T_{reg}, (g,a) \mapsto (g\dot{w}, \dot{w}^{-1}a\dot{w})$. 

We define $\mathcal{H}_{\Ll} \defeq \End(\delta_*(\tilde{\Ll}))$. For each $w \in W$, we consider the action of $w$. We set $$\mathcal{H}_{{\Ll},w} \defeq \Hom(\bar{\varphi}_w^*\left(\tilde{\Ll}\right), \tilde{\Ll}).$$ Since $\Ll$ is irreducible, this vector space is non-trivial (and hence has dimension $1$) if and only if $\bar{\varphi}_w^*\left( \tilde{\Ll}\right)\cong  \tilde{\Ll}$ if and only if $\ad(\dot{w}^{-1})^*(\Ll) \cong \Ll$.  
There is a natural pairing for $w,w' \in W_{\Ll}$:
\begin{align*}
	\mathcal{H}_{{\Ll},w} &\times \mathcal{H}_{{\Ll},w'} \to \mathcal{H}_{{\Ll},ww'} \\
	f \times g &\mapsto g \circ \bar{\varphi}_{w'}^*(f).
\end{align*}
Since $\delta_*(\tilde{\Ll}) \cong \delta_*\left(\bar{\varphi}_w^*(\tilde{\Ll})\right)$, we obtain a natural algebra isomorphism $$\mathcal{H}_{\Ll} \cong \bigoplus_{w \in W_{\Ll}}\mathcal{H}_{{\Ll},w}.$$

For $w \in W_{\Ll}$, we fix the unique isomorphism $\phi^{\Ll}_w: \ad(w^{-1})^*{\Ll}\to {\Ll}$ such that $(\phi^{\Ll}_w)_1$ is the identity. It induces an isomorphism $\tilde{\phi}^{\Ll}_w \in \mathcal{H}_{{\Ll},w}$. Observe that $\tilde{\phi}^\Ll_{ww'} = \tilde{\phi}^\Ll_{w} \times \tilde{\phi}^\Ll_{w'}$ for all $w,w' \in W_\Ll$. Therefore, $\mathcal{H}_\Ll \cong \Ql[W_\Ll]$ as algebras. Thanks to this construction, we can write \[IC(\G,\delta_*(\tilde{\Ll}))[\dim \G] \cong \bigoplus_{E \in \irr(W_\Ll)} \Aa^\Ll_E \otimes V_E,\] where $\Aa^\Ll_E$ are the irreducible character sheaves in the induction of $K$ to $\G$ and $V_E = \Hom(\Aa^\Ll_E,\Ind_\B^\G(K))$ are this time seen as the irreducible $\Ql[W_\Ll]$-modules.\\

Moreover, we get an identification as follows 
$$\mathcal{H}_{{\Ll},w} = \Hom(\ad(w^{-1})^*\tilde{{\Ll}},\tilde{{\Ll}}) = \Ql\tilde{\phi}^{\Ll}_w \longleftrightarrow \Ql\tilde{\phi}^{\Ql}_w = \mathcal{H}_{\Ql,w}\defeq \Hom(\bar{\varphi}_w^*\left(\Ql\right),\Ql).$$
This embedding of algebras $\mathcal{H}_{\Ll} \to \mathcal{H}_{\Ql}$ leads to a canonical isomorphism \cite[2.6.e]{lusztigCharacterValuesFinite1986}:
\[(\Aa^\Ll_E)_{\vert \G_{uni}} \cong \bigoplus_{E' \in \irr(W)} \langle E, \Res_{W_\Ll}^W E' \rangle (\Aa_{E'}^{\Ql})_{\vert \G_{uni}}.\]
\subsection{Central translation of character sheaves with unipotent support}\label{sect_centraltrans}
We would like to consider translation of character sheaves by an element of the center of $\G$. For $z \in  Z(\G)$, we define the translation $z: \T \to \T,\, t \mapsto zt$ for $t \in \T$.

\subsubsection{Translation of Kummer systems}
We want to get a better understanding of the isomorphism $\phi^\Ll_w : \ad(w^{-1})^*\Ll \to \Ll$ for $\Ll \in \mathcal{S}(\T)$ and $w \in W_\Ll$.

We fix $n \in \N$ and $c \in k^\times$. The stalk $\left(\left(\rho_n\right)_*\Ql\right)_c$ can be seen as the $n$-dimensional $\Ql$-vector space $\Ql[\rho_n^{-1}(c)]$, with action of $\mu_n$  on $\rho_n^{-1}(c)$ by multiplication. In that setting, $\left(\mathcal{E}_{n,\psi}\right)_c$ is the $\Ql$-vector subspace of dimension one on which the action of $x \in \mu_n$ is simply multiplication by $\psi(x)$. \\
We start by a few observations, defining some morphisms and keeping track of their restriction to the stalks. Fix $\lambda,\gamma \in X(\T)$. Firstly, we have a $\mu_n$-equivariant morphism: \[(\lambda\gamma)^*\left(\rho_n\right)_*\Ql \to (\lambda)^*\left(\rho_n\right)_*\Ql \otimes_{\mu_n}(\gamma)^*\left(\rho_n\right)_*\Ql,\]
On the stalk $t \in \T$, we get a morphism of $\mu_n$-modules from $\Ql[\rho_n^{-1}(\lambda(t)\gamma(t))]$  to $\Ql[\rho_n^{-1}(\lambda(t))] \otimes_{\mu_n} \Ql[\rho_n^{-1}(\gamma(t))].$\\ 
Now, we denote by $\underline{\Ql[\mu_n]}_V$ the constant sheaf on a variety $V$. The adjunction $\left(\rho_n\right)^*\left(\rho_n\right)_*\Ql\to \Ql$ is given by the $\mu_n$-equivariant isomorphism: \[\left(\rho_n\right)^*\left(\rho_n\right)_*\Ql \to \underline{\Ql[\mu_n]}_{k^\times}.\] Taking the pullback by $\gamma$, we get a $\mu_n$-equivariant morphism \[(\gamma^n)^*\left(\rho_n\right)_*\Ql = \gamma^*\left(\rho_n\right)^*\left(\rho_n\right)_*\Ql \to \gamma^*\underline{\Ql[\mu_n]}_{k^\times} \to \underline{\Ql[\mu_n]}_{\T}. \] 
On the stalk at $t \in \T$, we get an isomorphism of $\mu_n$-modules from $\Ql[\rho_n^{-1}(\gamma^n(t))] = \Ql[\gamma(t)\mu_n]$  to $\Ql[\mu_n],$ by multiplication by $\gamma^{-1}(t)$. \\
Combining the two previous $\mu_n$-equivariant morphisms, we get 
\[(\lambda\gamma^n)^*\left(\rho_n\right)_*\Ql \to (\lambda)^*\left(\rho_n\right)_*\Ql \otimes_{\mu_n}(\gamma^n)^*\left(\rho_n\right)_*\Ql \to (\lambda)^*\left(\rho_n\right)_*\Ql \otimes_{\mu_n}\underline{\Ql[\mu_n]}_{\T} \to (\lambda)^*\left(\rho_n\right)_*\Ql.\]
On the stalk $t \in \T$, we get a morphism of $\mu_n$-modules from $\Ql[\rho_n^{-1}(\lambda(t)\gamma^n(t))] = \Ql[\gamma(t)\rho_n^{-1}(\lambda(t))]$ to $\Ql[\rho_n^{-1}(\lambda(t))],$ given by multiplication by $\gamma^{-1}(t)$. \\
This  $\mu_n$-equivariant morphism restricts to an isomorphism:
\[\nu_{\lambda,\gamma,n}: (\lambda\gamma^n)^*\Ee_{n,\psi} \to \lambda^*\Ee_{n,\psi}.\]  Let  $t \in \T$ such that $t^n =1$. The definition of $\Ee_{n,\psi}$ implies that the isomorphism $\left(\nu_{\lambda,\gamma,n}\right)_t$ of $\mu_n$-modules from $\left(\Ee_{n,\psi}\right)_{\lambda(t)} \subseteq \Ql[\rho_n^{-1}(\lambda(t))]$ to $\left(\Ee_{n,\psi}\right)_{\lambda(t)} \subseteq \Ql[\rho_n^{-1}(\lambda(t))]$ is in fact given by multiplication by $\psi(\gamma(t)^{-1})$. More generally, since we always have $\lambda^*\Ee_{n,\psi}$ isomorphic to $(\lambda^m)^*\Ee_{nm,\psi}$ for all $m \in\N$, for any element $t \in \T$ of finite order, the isomorphism $\left(\nu_{\lambda,\gamma,n}\right)_t$ of $\mu_n$-modules is simply  multiplication by $\psi(\gamma(t)^{-1}).$\\

We can now consider the particular case of the isomorphism between $ \ad(w^{-1})^*\Ll \to \Ll$ for $\Ll \in \mathcal{S}(\T)$ and $w \in W_\Ll$. The above discussion leads to the following result:
\begin{lem}\label{lem_defisoadw}
	Let $\Ll = \lambda^*\Ee_{n,\psi} \in \mathcal{S}(\T)$ for $n \in \N$ and $\lambda \in X(\T)$. Let $w \in W_\Ll$. Recall that there is $\lambda_w \in X(\T)$ such that $\lambda^w= \lambda\lambda_w^n$. Then $\phi_w^\Ll$ in in fact $\nu_{\lambda,\lambda_w,n}$. 
\end{lem}
\begin{proof}Since $\Ll$ is irreducible, it suffices to check that $ \left(\nu_{\lambda,\lambda_w,n}\right)_1$ is equal to $\left(\phi_w^\Ll\right)_1$, which is the identity. It clearly follows from the previous discussion.  
\end{proof}

\begin{lem}\label{lem_transphiLl}
	Let $z \in Z(\G)$, $w \in W_{\Ll}$, then there is $\lambda_w \in \Hom(\T, k^\times)$ depending on $\Ll$ such that
	$$\psi(\lambda_w(z))z^*(\phi^{{\Ll}}_w) = \phi^{z^*{\Ll}}_w \quad \text{ and } \,  \psi(\lambda_w(z))z^*(\tilde{\phi}^{{\Ll}}_w) = \tilde{\phi}^{z^*{\Ll}}_w.$$ 
\end{lem}
\begin{proof}
	We have $W_{z^*{\Ll}} = W_{{\Ll}}$ and $\phi^{z^*{\Ll}}_w$ is well-defined. 
	Since $z^*(\phi^{{\Ll}}_w) \in \mathcal{H}_{z^*{\Ll},w}$, the two isomorphisms $z^*(\phi^{{\Ll}}_w)$ and $\phi^{z^*{\Ll}}_w$ differ by a scalar. To determine it, it suffices to consider the stalks at $1$. On one hand, by definition, $\left(\phi^{z^*{\Ll}}_w\right)_1$ is the identity. On the other hand, by Lemma \ref{lem_defisoadw}, $\left(z^*(\phi^{{\Ll}}_w)\right)_1 = \left(\nu_{\lambda,\lambda_w,n}\right)_1$ is given by multiplication by $\psi(\lambda_w(z))^{-1}$ and we can conclude. 
\end{proof}

\begin{rmk}
	Observe that the map $W_\Ll \to \Ql^\times, w \mapsto \psi(\lambda_w(z))$ is a group homomorphism since $z^*\phi^{{\Ll}}_{w'}\circ  \ad((\dot{w}')^{-1})^*z^*(\phi^{{\Ll}}_{w}) = z^*\left(\phi^{{\Ll}}_{w'}\circ  \ad((\dot{w}')^{-1})^*(\phi^{{\Ll}}_{w})\right) = z^*\phi^{{\Ll}}_{ww'}$ for all $w,w' \in W_\Ll.$ Thus, $\psi(\lambda_{w'}(z))^{-1} \circ \psi(\lambda_{w}(z))^{-1} = \psi(\lambda_{ww'}(z))^{-1}.$ 
\end{rmk}

We gather different facts on $\psi(\lambda_w(z))$ from \cite[$\S$ 11.8]{lusztigCharacterSheavesIII1985}:
\begin{lem}\label{lem_valchiz}
	Let $z \in Z(\G)$, ${\Ll} = \lambda^*(\mathcal{E}_{n,\psi})$. Then 
	\begin{enumerate}
		\item $\psi(\lambda_w(z)) = 1$ if $z \in Z^\circ(\G)$ or $w \in W^\circ_{\Ll}$,
		\item the map $W_{\Ll}/W^\circ_{\Ll} \to \Hom(Z(\G)/Z^\circ(\G), \, \Ql^\times)$, $w \mapsto \left(\bar{z} \to \psi(\lambda_w(z))\right)$ is injective. 
	\end{enumerate}
\end{lem}

\subsubsection{Translation of character sheaves} 
\begin{lem}\label{lem_translationCS}
	For $z \in Z(\G)$, $E \in \irr(W_{\Ll})$ we have 
	\[A^{{\Ll}}_E \cong A^{z^*{\Ll}}_{E\cdot \chi^z},\]
	where $\chi^z: W_{\Ll} \to \Ql^\times$ is given by $w \mapsto \psi(\lambda_w(z))$. 
\end{lem}

\begin{proof}	
	For $w \in W_{\Ll}$, we can define two embeddings of  $\mathcal{H}_{z^*{\Ll},w}$ into $\mathcal{H}_{\Ql,w}$: the first one being the usual one identifying $\phi^{z^*{\Ll}}_w$ with $\phi^{\Ql}_w$ and the second one identifying $z^*(\phi^{{\Ll}}_w)$ with $\phi^{\Ql}_w$. So we get two isomorphisms between $\Ql\left[ W_{\Ll}\right] $ and $ \End(IC(\G,\delta_*(\tilde{z^*{\Ll}}))).$ From Lemma \ref{lem_transphiLl}, we see that they differ by multiplication with $\psi(\lambda_w(z))$. We define $\chi^z: W_\Ll \to \Ql$ by $w \mapsto \psi(\lambda_w(z))$ for any $w \in W_\Ll$.  On one hand we have, 
	\begin{align*}
		A^{{\Ll}}_E &\cong \Hom_{\End(IC(\G,\delta_*(\tilde{{\Ll}})))}\left(E, IC\left(\G,\delta_*(\tilde{{\Ll}})\right)\right)&\\
		&\cong \Hom_{W_{\Ll}}\left(E, IC\left(\G,\delta_*(\tilde{{\Ll}})\right)\right) \quad &(\phi^{{\Ll}}_w \leftrightarrow \phi^{\Ql}_w)\\
		& \cong z^*\Hom_{W_{\Ll}}(E,z^*IC(\G,\delta_*(\tilde{{\Ll}}))) \quad &(z^*\phi^{{\Ll}}_w \leftrightarrow \phi^{\Ql}_w) \\
		& \cong \Hom_{W_{\Ll}}(E,IC(\G,\delta_*(\tilde{z^*{\Ll}}))). &
	\end{align*}
	
	\noindent On the other hand, for $E' \in \irr(W_\Ll)$, we have 
	\begin{align*}
		A^{z^*{\Ll}}_{E'} &\cong z^*\Hom_{\End(IC(\G,\delta_*(\tilde{z^*{\Ll}})))}\left(E', IC\left(\G,\delta_*(\tilde{z^*{\Ll}})\right)\right)&\\
		&\cong \Hom_{W_{\Ll}}(E',IC(\G,\delta_*(\tilde{z^*{\Ll}})))  &(\phi^{z^*{\Ll}}_w = \chi^z(w)z^*\phi^{{\Ll}}_w \leftrightarrow \phi^{\Ql}_w)	\\
		&\cong \Hom_{W_{\Ll}}(E'\cdot (\chi^z)^{-1},IC(\G,\delta_*(\tilde{z^*{\Ll}}))).  &(z^*\phi^{{\Ll}}_w \leftrightarrow \phi^{\Ql}_w)	
	\end{align*}
	We conclude that $A^{{\Ll}}_E\cong A^{z^*{\Ll}}_{E'}$ if and only if $E' = E \cdot \chi^z.$
\end{proof}



\subsection{Restriction of an induced cuspidal perverse sheaf to a mixed conjugacy class}
We fix a character sheaf $K \defeq IC(\T, \Ll)$ on the torus  where $\Ll = \lambda^*(\Ee_{n,\psi}) \in \mathcal{S}(\T)$. We also fix $s \in \G$ a semisimple element. 
For any $g \in \G$, we define the left translation $g: \G \to \G, h \mapsto gh$. 
To simplify notation, we set ${\H} \defeq  C^\circ_\G(s)$. We also write $W^{{\H}} \defeq N_{{\H}}(\T )/\T $.

\begin{dfn}
	We let $$M \defeq \{m \in \G \mid m^{-1}s m \in \T\} \text{ and } \Gamma \defeq \H \backslash M/\T .$$
\end{dfn}

For each $\gamma \in \Gamma$, we define \[\T_\gamma \defeq \dot{\gamma}\T\dot{\gamma}^{-1} \cap \H \subseteq \dot{\gamma}\B\dot{\gamma}^{-1} \cap \H =: \B_\gamma.\]
We set $\mathcal{L}_\gamma$ the local system on $\T_\gamma$ obtained as the inverse image of $\Ee$ under the map $\tau_\gamma: \T_\gamma \to \T, g \mapsto \dot{\gamma}^{-1}s g\dot{\gamma}$. By \cite[Prop.~7.11]{lusztigCharacterSheavesII1985} we can then define irreducible cuspidal character sheaves on $\T_\gamma$, \[K_\gamma \defeq IC(\T_\gamma,\Ll_\gamma)[\dim \T_\gamma].\] The induced perverse sheaf $\Ind_{\B_\gamma}^{\H}(K_\gamma)$ is semisimple and decomposes into a direct sum of character sheaves on $\H$. 

\begin{rmk}
	We will often abuse notation and write only $\gamma$ for $\dot{\gamma}$ for any $\gamma \in \Gamma$. 
\end{rmk}

\begin{rmk}
	Let $\gamma \in \Gamma$. We set $\B_{0, \gamma} \defeq \B \cap \H^\gamma.$
	
	We also fix $\Ll_{0,\gamma}$ the local system on $\T$ obtained as the inverse image of $\Ll$ under the map $\T\to \T, g \mapsto \gamma^{-1}s\gamma g$. We define the perverse sheaf on $\T$, $$K_{0,\gamma} \defeq IC(\T,\Ll_{0,\gamma})[\dim \T].$$
	We observe that ${\Ll}_\gamma = \tau_\gamma^*{\Ll} = \ad(\gamma^{-1})^*( \gamma^{-1}s\gamma)^*{\Ll}=\ad(\gamma^{-1})^*{\Ll}_{0,\gamma}.$ 
	Therefore,
	\[\Ind_{\B_{\gamma}}^{\H}(K_{\dot{\gamma}}) \cong \ad(\dot{\gamma}^{-1})^*\left(\Ind_{\B_{0,\gamma}}^{\H^{\dot{\gamma}}}(K_{0,\gamma})\right).\]
\end{rmk}

By \cite[Prop. 8.2.3]{AST_1989__173-174__111_0} and \cite[$\S 8$]{lusztigCharacterSheavesII1985}, we can decompose  $s^*((\Ind_\B^\G(K))_{\vert s\H_{uni}})$ into a direct sum of the various $(\Ind_{\B_{\gamma}}^{\H}(K_{\gamma}))_{\H_{uni}}$ for $\gamma \in \Gamma$:
\begin{prop}[{\cite[$\S 8$]{lusztigCharacterSheavesII1985}}]\label{prop_iso} There is an open neighborhood $U$ of $s$ in ${\H}$, such that $s {\H}_{uni} \subseteq U$ and
	\[s^*((\Ind_\B^\G(K))_{\vert U}) \cong \bigoplus_{\gamma \in \Gamma}(\Ind_{\B_{\gamma}}^{\H}(K_{\gamma}))_{s^{-1}U}[\dim(\G)-\dim(\H)].\]
\end{prop}

We describe the isomorphism on the level of local systems thanks to the proof of \cite[Prop. 8.2.3]{AST_1989__173-174__111_0} and the discussion following it.\\

For each $\gamma \in \Gamma$, we have the following commutative diagram:
\begin{center}
	\begin{tikzcd}
		\T_{\gamma} \arrow[d, "\tau_{\gamma}"'] & \H \times \T_{{\gamma} ,reg} \arrow[l, "\alpha_{\gamma}"'] \arrow[r, "\beta_{\gamma}"] \arrow[d] & \H \times_{\T_{\gamma}} \T_{{\gamma}, reg} \arrow[r, "\delta_{\gamma}"] \arrow[d, "s_{\gamma}"'] & {Y_{\T_{\gamma}}} \arrow[d, "s"] \\
		\T   & \G\times \T_{reg} \arrow[l, "\alpha"] \arrow[r, "\beta"']                                       & \G\times_{\T}\T_{reg} \arrow[r, "\delta"']                               & {Y_{\T}}                   
	\end{tikzcd}
\end{center}

with \begin{itemize}
	\item the sets $\T_{reg}$, $\T_{\gamma,reg} \defeq \{g \in \T_{\gamma} \mid C^\circ_{\H}(g) \subseteq \T_{\gamma} \}$, $Y_{\T} \defeq \bigcup_{g \in G}g\T_{reg}g^{-1}$ and $Y_{\T_{\gamma}} \defeq \bigcup_{h \in C^\circ_G(\T)}h\T_{\gamma, reg}h^{-1}$,
	\item the map $\alpha$ (resp.~$\alpha_{\gamma}$) the projection on $\T$ (resp.~$\T_{\gamma})$ of the second coordinate  ,
	\item the map $\beta$ (resp. $\beta_{\gamma}$) induced by quotienting,
	\item the maps $\delta: (g,a) \mapsto gag^{-1}$ for $g \in \G, a \in \T_{reg}$ and $\delta_{\gamma}: (g,a) \mapsto gag^{-1}$ for $g \in \G$ and $ a \in \T_{\gamma, reg}$,
	\item the map $\tau_{\gamma} : g \mapsto \dot{\gamma}^{-1}s g\dot{\gamma}$,
	\item the map $s_{\gamma} : (h,g) \mapsto (h\dot{\gamma}, \dot{\gamma}^{-1}s g\dot{\gamma})$, for $h \in \H$ and $g \in \T_{\gamma, reg}$
	\item and the map $s : g \mapsto s g$ for $g \in Y_{\T_{\gamma}}$.
\end{itemize}
We define $Q \defeq \delta^{-1}(U \cap \overline{Y_{\T}})$ and $S \defeq Q \cap \G\times_{\T}\T_{reg}$. For each $\gamma \in \Gamma$, we set $Q_{\gamma} \defeq \{(g,g') \in Q \mid g \in \gamma\}$ and $S_\gamma \defeq S\cap Q_\gamma = \G\times_{\T}\T_{reg} \cap Q_\gamma$. 
We observe that $s_{\gamma}$ induces an isomorphism from 
$\delta_{\gamma}^{-1}(s^{-1} U\cap Y_{\T_{\gamma}})$ to $Q_{\gamma}$. There exists a local system $\tilde{\Ll}$ on $\G\times_{\T}\T_{reg}$ such that $\alpha^*(\Ll) = \beta^*(\tilde{\Ll})$. We set $\tilde{\Ll}_{\gamma} \defeq s^*_{\gamma}(\tilde{\Ll})$. We then have \[\alpha_{\gamma}^*\left(\Ll_{\gamma}\right) = \alpha_{\gamma}^*\tau_\gamma^*\left(\Ll\right) \cong \beta_\gamma^*s_\gamma^*\left(\tilde{\Ll}\right) = \beta_{\gamma}^*(\tilde{\Ll}_{\gamma}).\]
Moreover, by the change of basis theorem, since the diagram is commutative, 
\begin{align*}
	s^*(\delta_*(\tilde{\Ll})_{\vert \delta(S_{\gamma})}) &\cong (\delta_{\gamma})_*((s_{\gamma})^*(\tilde{\Ll}_{\vert S_\gamma}))) \\
	&\cong (\delta_{\gamma})_*((s_{\gamma})^*(\tilde{\Ll}_{\vert \G\times_{\T}\T_{reg} \cap s_\gamma(\delta_\gamma^{-1}(s^{-1}U\cap Y_{\T_{\gamma}}))}))) \\
	&\cong (\delta_{\gamma})_*((\tilde{\Ll}_{\gamma})_{\vert \H\times_{\T_\gamma}\T_{\gamma reg} \cap \delta_{\gamma}^{-1}(s^{-1}U\cap Y_{\T_{\gamma}}))}) \\
	&\cong (\delta_{\gamma})_*((\tilde{\Ll}_{\gamma})_{\vert  \delta_{\gamma}^{-1}(s^{-1}(U \cap Y_{\T_{\gamma}}))}) \\
	&\cong ((\delta_{\gamma})_*(\tilde{\Ll}_{\gamma}))_{\vert s^{-1}(U \cap Y_{\T_{\gamma}})}.
\end{align*}
We then have $U\cap Y_{\T} = \bigcup_{\gamma \in \Gamma}\delta_{\gamma}\left(S_{\gamma}\right)$ (\cite[$\S$ 8.2.17]{lusztigCharacterSheavesII1985}) and a canonical isomorphism \[s^*(\delta_*(\tilde{\Ll})_{\vert U\cap Y_{\T}}) \cong \bigoplus_{\gamma \in \Gamma}((\delta_{\gamma})_*(\tilde{\Ll}_{\gamma}))_{\vert s^{-1}(U \cap Y_{\T_{\gamma}})}.\]
Observe that $\Ind_{\B_\gamma}^\H(K_{\gamma}) = IC(\overline{Y_{\T_{\gamma}}}, {\delta_{\gamma}}_*\tilde{\Ll}_{\gamma})[\dim Y_{\T_{\gamma}}] = IC(\H, {\delta_{\gamma}}_*\tilde{\Ll}_{\gamma})[\dim \H].$ The isomorphism above gives rise to an isomorphism 
\[s^*((\Ind_\B^\G(K))_{\vert U \cap Y_{\T}}) \cong \bigoplus_{\gamma \in \Gamma}(\Ind_{\B_{\gamma}}^{\H}(K_{\gamma}))_{s^{-1}(U\cap Y_{\T})}[\dim(\G)-\dim(\H)].\]
By \cite[8.8.4--8.8.7]{lusztigCharacterSheavesII1985}, this isomorphism can be uniquely extended to the isomorphism in Proposition \ref{prop_iso}: 
\[s^*((\Ind_\B^\G(K))_{\vert U}) \cong \bigoplus_{\gamma \in \Gamma}(\Ind_{\B_{\gamma}}^{\H}(K_{\gamma}))_{s^{-1}U}[\dim(\G)-\dim(\H)].\]

Observe that $s^*((\Ind_\B^\G(K))_{\vert U})$ might not be a semisimple perverse sheaf. However, Lusztig showed in \cite[Prop.~1.4]{lusztigRestrictionCharacterSheaf2015} that
\[s^*((\Ind_\B^\G(K))_{\vert s\H_{uni}})[-\dim(\G) + \dim(\H)-\dim(\T)]\cong \bigoplus_{\gamma \in \Gamma}(\Ind_{\B_{\gamma}}^{\H}(K_{\gamma}))_{s^{-1}\H_{uni}}[-\dim(\T)],\] 
is indeed semisimple.

\begin{rmk}
Observe that the right hand side of the isomorphism does not depend on $s$ but rather on its conjugacy class. Therefore, from now on, we will assume that $s \in \T$. 
\end{rmk}

\subsection{Restriction of a character sheaf to a mixed conjugacy class}

By Proposition \ref{prop_iso}, we know how to decompose the restriction of the induction of a cuspidal character sheaf $K$ to a mixed conjugacy class as a direct sum of inductions of cuspidal character sheaves on a smaller group. The main goal is to understand how this decomposition behaves with respect to the action of the endomorphism algebra of $K$ or rather the relative Weyl group associated to $K$.
\subsubsection{Action of $W_{{\Ll}}$ on $\Gamma$ }

As in \cite{lusztigRestrictionCharacterSheaf2015}, we define an action of $W_{{\Ll}}$ on $\Gamma$. 
\begin{dfn} We define an action of $N_\G(\T, \Ll)$ on $M$ by  $nm \defeq  mn^{-1}$ for all $m \in M$ and $n \in N_\G(\T, \Ll)$. It induces a well defined action of $W_{{\Ll}}$ on $\Gamma$ by  $w.\gamma = w.{\H}\dot{\gamma}\T  \defeq {\H} \dot{\gamma}\dot{w}^{-1}\T $ for all $\gamma \in \Gamma$, $w \in W_{{\Ll}}$.
\end{dfn}

\begin{dfn}
	We fix a set $\Lambda$ of orbit representatives for the action of $W_{{\Ll}}$ on $\Gamma$, \[\Gamma = \bigsqcup_{\lambda \in \Lambda}W_{{\Ll}}.\lambda. \]
\end{dfn}

\begin{lem}
	Let $w \in W_{{\Ll}}$ and $\lambda\in \Lambda$. Then \[\T_{w.\lambda} = \T_{\lambda}. \]
	Moreover, \[\Ll_{w.\lambda} \cong \Ll_{\lambda} \text{ and } \, \Ll_{0, w.\lambda} \cong \ad(w^{-1})^*\Ll_{0, \lambda}.\]
\end{lem}
\begin{proof}
	It follows from the definition of $W_{\Ll}$ and from the fact that $\tau_{w.\lambda} = \ad(w) \circ \tau_{\lambda}$.
\end{proof}


Let $\gamma \in \Gamma$, then the stabilizer of $\dot{\gamma}$ by the action of $N_\G(\T,\Ll)$ is $N_\G(\T,\Ll) \cap \H^\gamma = N_{\H^\gamma}(\Ll)$. Thus, the stabilizer of $\gamma$ by the action of $W_{{\Ll}}$ is $N_{\H^\gamma}(\T,\Ll)/\T$. We would like to better understand this group. To ease notation, we write $W^\gamma_\Ll \defeq N_{\H^\gamma}(\T,\Ll)/\T.$

\begin{lem}\label{lem_valueinterstW}
	Let $\gamma \in \Gamma$, then $$N_{\H^\gamma}(\Ll)  = N^{\H^\gamma}_{\Ll_{0,\gamma}} \quad  \text{and} \quad W^\gamma_\Ll = W^{\H^\gamma}_{\Ll_{0,\gamma}}.$$
\end{lem}

\begin{proof}
	Let $h \in N_{\H^\gamma}(\Ll)$. Recall that $\Ll_{0,\gamma} = (\gamma^{-1}s\gamma)^*\Ll$ and $h \in C_\G(\gamma^{-1}s\gamma).$ We therefore have
	\[\ad(h)^*\Ll_{0,\gamma} = \ad(h)^*(\gamma^{-1}s\gamma)^*\Ll \cong  (\gamma^{-1}s\gamma)^*\ad(h)*\Ll \cong (\gamma^{-1}s\gamma)^*\Ll = \Ll_{0,\gamma}. \] 
	Thus, $h \in  N^{\H^\gamma}_{\Ll_{0,\gamma}}.$
	On the other hand, let $h \in  N^{\H^\gamma}_{\Ll_{0,\gamma}}$. Symmetrically we have $\Ll \cong (\gamma^{-1}s^{-1}\gamma)^*\Ll_{0,\gamma}$ and $h \in C_\G(\gamma^{-1}s^{-1}\gamma)$ and we conclude that $h \in N_{\H^\gamma}(\Ll).$ The claims follow. 
\end{proof}

\begin{lem}\label{lem_orbits}
	There is a bijection $\omega$ between $W^{{\H}}\backslash W/ W_{{\Ll}}$ and the $W_{{\Ll}}$-orbits on~$\Gamma.$  
\end{lem}
\begin{proof} We define a $W_{{\Ll}}$-equivariant bijection\begin{align*} \omega:
		W^{{\H}}\backslash W &\to \Gamma\\
		W^{{\H}}w &\mapsto  {\H}\dot{w}\T .
	\end{align*} 
	We first show that the map is well-defined. Fix $W^{{\H}}w \in W^{{\H}}\backslash W$ and let $n_1 \in N_\G(\T )$ and $n_2 \in N_{{\H}}(\T )$ such that $n_2\T  n_1\T = w$. Then, there is $t \in \T $ such that $n_2n_1t = \dot{w}$. Then we see that $$\omega(W^{{\H}}w) = {\H}n_2n_1t\T    ={\H}n_1\T= \omega(W^{{\H}}n_1\T ).$$  Observe that $\omega$  is $W_{{\Ll}}$-equivariant. Indeed, let's fix $W^{{\H}}w$ in $W^{{\H}}\backslash W$ and $v \in W_{{\Ll}}$.
	Then \[\omega(W^{{\H}}wv^{-1}) =  {\H}\dot{w}\dot{v}^{-1}\T  =  v.{\H}\dot{w}\T = v.\omega(W^{{\H}}w).\]
	To show surjectivity, we fix $m \in M$.  We observe that $^{m}\T  \subseteq \H$ and $\T  \subseteq \H$. Therefore, there is $h \in \H$ such that $^{m}\T  =\leftindex^h{\T}.$ Thus, $h^{-1}m \in N_\G(\T )$. Let $w \in W$ such that $w = h^{-1}m\T$, i.e. $\dot{w} =h^{-1}mt$ for some $t \in \T $. Then $\omega(W^{{\H}}w) = {\H}\dot{w}\T  = {\H}h^{-1}mt\T = {\H}m\T $ and $\omega$ is surjective. Injectivity is clear. 
\end{proof}

\subsubsection{Restriction of a character sheaf to a mixed conjugacy class}

We define an action of $W_{{\Ll}}$ on the right hand side of Proposition \ref{prop_iso}, so that it commutes with the isomorphism.\\
We write $Y \defeq Y_{\T}$ and $Y_\gamma \defeq Y_{\T\gamma}$ for each $\gamma \in \Gamma$. 
For each $\lambda \in \Lambda$, we fix a set $V^\lambda$ of left cosets representatives  such that $W_{\Ll} \defeq \bigsqcup_{v \in V^\lambda} W^\lambda_\Ll v$. In other words, we have \[W_{\Ll}.\lambda = \{v.\lambda \mid v \in V^\lambda\}.\]
Proposition \ref{prop_iso} becomes \[s^*(IC(\overline{Y}, \delta_*\tilde{\Ll})_{\vert U})\cong \bigoplus_{\lambda \in \Lambda}\bigoplus_{v \in V^\lambda} IC(\overline{Y}_\lambda, \delta_*s_{v.\lambda}^*\tilde{\Ll})_{s^{-1}U}.\]
Since for any $\lambda \in \Lambda$ and $v \in V^\lambda$, $s_{v.\lambda} = \bar{\varphi}_{v^{-1}} \circ s_\lambda$, we in fact have 
\[s^*(IC(\overline{Y}, \delta_*\tilde{\Ll})_{\vert U}) \cong \bigoplus_{\lambda \in \Lambda}\bigoplus_{v \in V^\lambda} IC(\overline{Y}_\lambda, \delta_*s_{\lambda}^*\bar{\varphi}_{v^{-1}}^*\tilde{\Ll})_{s^{-1}U}.\]

By \cite[above lemma 8.6]{AST_1989__173-174__111_0} or \cite[$\S$ 8.7.13]{lusztigCharacterSheavesII1985}, we have that $\delta(S_\gamma) = \delta(S_{w.\gamma})$ for all $\gamma \in \Gamma$ and $w \in W_{\Ll}$.
Fix $w \in W_{{\Ll}}$. For each $\lambda \in \Lambda$ and $v \in V^\lambda$, there is $v' \in V^\lambda$ and $w_0 \in W^\lambda_\Ll$ such that $wv = w_0v'.$ Then we get
\begin{align*}
	s^*(IC(\overline{Y}, \delta_*\bar{\varphi}_{w^{-1}}^*\tilde{\Ll})_{\vert U}) &\cong \bigoplus_{\lambda \in \Lambda}\bigoplus_{v \in V^\lambda} IC(\overline{Y}_\lambda, \delta_*s_{\lambda}^*\bar{\varphi}_{v^{-1}}^*\bar{\varphi}_{w^{-1}}^*\tilde{\Ll})_{s^{-1}U} \\
	& \cong \bigoplus_{\lambda \in \Lambda}\bigoplus_{v \in V^\lambda} IC(\overline{Y}_\lambda, \delta_*s_{\lambda}^*\bar{\varphi}_{(v')^{-1}}^*\bar{\varphi}_{w_0^{-1}}\tilde{\Ll})_{s^{-1}U}.
\end{align*}

Recall that we fixed an isomorphism $\tilde{\phi}^\Ll_w:  \bar{\phi}_w^*(\tilde{\Ll}) \to \tilde{\Ll}$ in Subsection \ref{sect_centraltrans}. 
Observe that the following diagram commutes: 

\begin{center}
	\begin{tikzcd}
		\left(s^*(\delta_*\bar{\varphi}_{w^{-1}}^*\tilde{\Ll})\right)_{\vert \delta(S_{\lambda})}  \arrow[r]    \arrow[d,"s^*\delta_*(\tilde{\phi}^\Ll_w)"'] & \left(\delta_*s_{\lambda}^*\bar{\varphi}_{(v')^{-1}}^*\bar{\varphi}_{w_0^{-1}}^*\tilde{\Ll}_{\vert S_\lambda}\right) \arrow[d,"\delta_*s_{\lambda}^*\bar{\varphi}_{(v')^{-1}}^*(\tilde{\phi}^\Ll_{w_0})"]\\
		\left(s^*(\delta_*\tilde{\Ll})\right)_{\vert \delta(Y_{\lambda})}  \arrow[r] & \left(\delta_*s_{\lambda}^*\bar{\varphi}_{(v')^{-1}}^*\tilde{\Ll}_{\vert S_\lambda}\right) 
	\end{tikzcd} 
\end{center}
where the horizontal lines are given by the canonical isomorphisms from the change of basis. 

Hence, we define an isomorphism  \[IC(\overline{Y}_\lambda, \delta_*s_{\lambda}^*\bar{\varphi}_{v^{-1}}^*\bar{\varphi}_{w^{-1}}^*\tilde{\Ll})_{s^{-1}U} \to IC(\overline{Y}_\lambda, \delta_*s_{\lambda}^*\bar{\varphi}_{(v')^{-1}}^*\tilde{\Ll})_{s^{-1}U}\] by rearranging the terms and acting on each component via $\delta_*s_{\lambda}^*\bar{\varphi}_{(v')^{-1}}(\tilde{\phi}^\Ll_{w_0})$. By definition, it commutes with the isomorphism in Proposition \ref{prop_iso}. \\

The set of maps $a^\lambda_{w} \defeq s^*_\lambda(\tilde{\phi}^\Ll_w)$ for $w  \in W^\lambda_\Ll = W^{\H^\lambda}_{\Ll_{0,\lambda}}$ lifts to a basis of $\mathcal{H}_{\Ll_{0,\lambda}}$ and therefore induces an algebra isomorphism between $\Ql[W^{\H^\lambda}_{\Ll_{0,\lambda}}]$ and $\End(IC(\overline{Y}, \Ll_{0,\lambda})).$ We would like to link the $a^\lambda_{ w }$ for $w \in W^\lambda_\Ll$ to the isomorphisms $ {\phi}_w^{\Ll_{0,\lambda}}$. Here $a^\lambda_w$ is simply the translation of $\phi^{\Ll_{0,\lambda}}_w$ by $s^\lambda \in Z(\H^\lambda)$. Therefore, we can apply Lemma \ref{lem_transphiLl} and we see that $$\psi(\lambda_w(s^\lambda))a^\lambda_w = \phi^{(s^\lambda)^*\Ll}_w.$$  
We finally obtain a formula for the restriction of a character sheaf. We write 
\[IC(\G, \tilde{\Ll})[\dim \G] \cong \bigoplus_{E \in \irr(W_{\Ll})} V_E \otimes \Aa_E \quad \text{and}\] \[ IC(\H^\lambda, \tilde{\Ll}_{0,\lambda})[\dim \H] \cong \bigoplus_{E' \in \irr(W_{\Ll_{0,\lambda}})} V_{E'} \otimes \Aa^\lambda_{E'} \quad \forall \lambda \in \Lambda. \] We let $\chi^\lambda: W^\lambda_\Ll \to \Ql^\times, w \mapsto  \psi(\lambda_w(s^\lambda)).$  
We let $d \defeq \dim(\H)- \dim(\T)$ and for any symbol $S$, any $ W^{{\H}}w W_{{\Ll}}$, we write $S^w \defeq S^{\omega(w)}$, for $\omega$ as in Lemma \ref{lem_orbits}.
 
\begin{prop}\label{prop_formula} 
	For $E \in \irr(W_{{\Ll}})$, let $\Aa_E$ as above. Then $s^*((\Aa_E)_{\vert s \H_{uni}})[d-\dim(\G)]$ is isomorphic to
	\[ \bigoplus_{w \in W^{{\H}}\backslash W/ W_{{\Ll}}} \bigoplus_{E' \in \irr(W^{{\H}^{w}})}	\langle \Res_{W^{w}_{\Ll}}^{W_{{\Ll}}} E \cdot\chi^{s^{w}},	\Res_{W^{w}_{\Ll}}^{W^{{\H}^{w}}}E'  \rangle_{W^{w}_\Ll} (\Aa^{\Ql}_{E'})_{\vert {\H}^{w}_{uni}}[d-\dim(\H)].\]
\end{prop}
\begin{proof} Firstly, we see that $	s^*((\Aa_E)_{\vert s {\H}_{uni}})[-\dim(\G) + d]$ is isomorphic to
	\begin{align*}
	 & \Hom_{\Ql[W_{\Ll}]}\left(V_E,	s^*\left(IC(\overline{Y}, \tilde{\Ll})_{\vert s {\H}_{uni}}[d]\right)\right)
		\\ &\cong \Hom_{\Ql[W_{\Ll}]}\left(V_E,\bigoplus_{\lambda \in \Lambda}\Ind_{\Ql[W^\lambda_\Ll]}^{\Ql[W_{\Ll}]} \Res_{\Ql[W^\lambda_\Ll]}^{\Ql[W_{\Ll_{\lambda}}]} IC(\overline{Y}_\lambda, \tilde{\Ll}_{\lambda})_{{\H}_{uni}}[d]\right)
		\\ &\cong \bigoplus_{\lambda \in \Lambda}\Hom_{\Ql[W_{\Ll}^\lambda]}\left(\Res_{\Ql[W^\lambda_\Ll]}^{\Ql[W_{\Ll}]}V_E, \Res_{\Ql[W^\lambda_\Ll]}^{\Ql[W_{\Ll_{\lambda}}]} IC(\overline{Y}_\lambda, \tilde{\Ll}_{\lambda})_{{\H}_{uni}}[d]\right)
		\\ &\cong \bigoplus_{\lambda \in \Lambda} \bigoplus_{V_{E'} \in \irr(\Ql[W_{\Ll_{\lambda}}])} \Hom_{\Ql[W_{\Ll}^\lambda]}\left(\Res_{\Ql[W^\lambda_\Ll]}^{\Ql[W_{\Ll}]}V_E, \Res_{\Ql[W^\lambda_\Ll]}^{\Ql[W_{\Ll_{\lambda}}]}V_{E'}\right) (\Aa^\lambda_{{E'}})_{{\H}_{uni}}[-\dim(\T)].
	\end{align*}
%
	Here the action of $W_\Ll^{\lambda}$ on $IC(\H^{\lambda}, \tilde{\Ll}_{0,{\lambda}})$ is induced by the maps $a^\lambda_v$ for $v \in W_\Ll^{\lambda}$. Now if instead we consider the usual action of $W_\Ll^{\lambda}$ on $IC(\H^{\lambda}, \tilde{\Ll}_{0,{\lambda}})$ induced by $\phi^{(s^{\lambda})^*\Ll}_v$, we are in a similar situation as in Lemma \ref{lem_translationCS}.
	We thus obtain $s^*((\Aa_E)_{\vert s {\H}_{uni}})[-\dim(\G) + d]$ is isomorphic to
	\begin{align*}
	\bigoplus_{w \in W^{{\H}}\backslash W/ W_{{\Ll}}} \bigoplus_{E' \in \irr(W^{w}_{\Ll})}	\langle \Res_{W^{w}_{\Ll}}^{W_{{\Ll}}} E \cdot\chi^{s^{w}},E'  \rangle_{W^{w}_\Ll} (\Aa^{\lambda}_{E'})_{\vert {\H}^{w}_{uni}}[-\dim \T].	
	\end{align*}
	By \cite[$\S$ 2.6]{lusztigCharacterValuesFinite1986},
	$\Ind_{\B}^{\H^{w}}(IC(\T ,\Ql))_{\vert \H^{w}_{uni}} \cong \Ind_{\B}^{\H^{w}}(IC(\T ,{\Ll}))_{\vert \H^{w}_{uni}},$ and we get 
	\[s^*((\Aa_E)_{\vert s \H_{uni}})[-\dim(\G) + d] \cong \bigoplus_{w } \bigoplus_{E'}	\langle \Res_{W^{w}_{\Ll}}^{W_{{\Ll}}} E \cdot\chi^{s^{w}},	\Res_{W^{w}_{\Ll}}^{W^{{\H}^{w}}}E'  \rangle_{W^{w}_\Ll} (\Aa^{\Ql}_{E'})_{\vert {\H}^{w}_{uni}}[-\dim(\T)],\]
	where $w$ runs over the double cosets representatives of $ W^{{\H}}\backslash W/ W_{{\Ll}}$ and $E'$ over the irreducible modules of $\Ql[W^{{\H}^w}]$. 
\end{proof}

\begin{rmk} It is usually more practical to see  $s^*((\Aa_E)_{\vert s \H_{uni}})[d-\dim(\G)]$ isomorphic to 
\[\bigoplus_{w } \bigoplus_{E'}	\langle \Res_{W^{w}_{\Ll}}^{W_{{\Ll}}} E \cdot\chi^{s^{w}},	\Res_{W^{w}_{\Ll}}^{W^{{\H}^{w}}}E'^{w}  \rangle_{W^{w}_\Ll} (\Aa^{\Ql}_{E'})_{\vert {\H}_{uni}}[-\dim(\T)],\]
where $w$ runs over the double cosets representatives of $ W^{{\H}}\backslash W/ W_{{\Ll}}$ and $E'$ over the irreducible modules of $\Ql[W^{{\H}}]$. 
\end{rmk}

\section{The two leftover cases in $E_8$}
We now have the tools to check the conditions in Proposition \ref{prop_excellent} for the two  $\ell$-special unipotent conjugacy classes of $\G$ for which we can not apply Corollary \ref{cor_genial}.
\begin{lem}
The unipotent classes in Proposition \ref{prop_casesleft} satisfy the conditions of Proposition \ref{prop_excellent}. 
\end{lem}
It suffices to show the existence of an admissible covering and the conditions in Proposition \ref{prop_DecKawExc}.\\

For the rest of this section, we assume that $\G$ is simple of type $E_8$. For each simple root $\beta \in \Delta$, we set $\omega_\beta$ the fundamental dominant coweight corresponding to $\beta$. We fix a bijection $\alpha$ from the semisimple elements of $\T^*$ to the Kummer local systems on $\T$. For $t \in \T^*$, by \cite[Prop.~4.4]{acharLocalisationFaisceauxCaracteres2010}, we have \[W_\Ll \cong W_{C_{\G^*}(t)}(\T^*).\]

\subsection{The unipotent conjugacy class $E_6(a_3) + A_1$ when $\ell = 3$} 
We fix the setting in the case where the $F$-stable unipotent class is $\Cc = E_6(a_3) + A_1$. We choose $t \in \T^*$ such that $C_{\G^*}(t)$ is of type $E_6 \times A_2$ and $v \in C_{\G^*}(t)$ lies in the unipotent class $A_2, 111$ of $E_6 \times A_2$. 

\subsubsection{Admissible covering}
We follow \cite[$\S$ 10.2]{brunatUnitriangularShapeDecomposition2020}. We can choose $s = \omega_{\alpha_1}(1/2)$ and $u_\Cc \in C_\G^\circ(s)$ $F$-fixed. In that case $C_\G(s)$ is of type $D_8$. We set $\H \defeq C_\G({s})$. Thanks to Algorithm 5.2 in \cite{brunatUnitriangularShapeDecomposition2020}, since only one unipotent class of $\H$ fuses into $\Cc$, we know that $u_\Cc$ lies in the unipotent class $6631$ of $\H$ which fuses into $E_6(a_3) + A_1$. 
The group  $A \defeq \langle {s} \rangle$ can be chosen as an admissible covering of $A_\G(u_\Cc)$ for a fixed cocharacter.  Observe that $$A \cong A_\G(u_\Cc) \cong \Omega_{C_{\G^*(t)},(v)_{C_{\G^*}(t)}}$$ and if $\ell = 3$, then $ A \cong \Omega^\ell_{\G,\Cc}.$

We are left to check the last condition of Proposition \ref{prop_DecKawExc}.
\subsubsection{Character sheaves} 
We fix $\Ll \defeq \alpha(t) \in \mathcal{S}(\T)$. We now consider a principal series character sheaf of $\G$ coming from $\Ll$ with unipotent support $E_6(a_3) + A_1$. Thanks to Lusztig's map ( \cite[Thm.~10.7]{lusztigUnipotentSupportIrreducible1992}), in CHEVIE notation, we choose the one corresponding to the character $\phi_{30,15},111$ of $W_\Ll$. Observe that $\Aa \defeq \Aa^\Ll_{\phi_{30,15},111}$ is $F$-stable.\\
Using CHEVIE, we compute that $W^{\H}\backslash W /W_{\Ll} = \{1,g\}$ for some $g \in W$. The groups $W^{\H} \cap W_{\Ll}$ and $W^{\H^g} \cap W_{\Ll} $ are both Weyl groups. Therefore $\chi^{s}$ and $\chi^{s^g}$ are trivial, by Lemma \ref{lem_valchiz}. We can thus use the Mackey formula to simplify Proposition \ref{prop_formula}. 
In that case, it becomes for any $E \in \irr(W_\Ll)$
\[s^*((\Aa_E)_{\vert s \H_{uni}})[-\dim(\G) + d] \cong  \bigoplus_{E' \in \irr(W^{{\H}})}	\langle \Ind^{W}_{W_{{\Ll}}} E ,	\Ind^{W}_{W^{{\H}}}E'  \rangle_{W} (\Aa^{\Ql}_{E'})_{\vert {\H}_{uni}}[-\dim(\T)].\]
We consider the restriction to $({s} u_\Cc)_\G$. By the same argument as in \cite[Proof of Thm.~2.4]{lusztigRestrictionCharacterSheaf2015}, we need to consider only the character sheaves of $\H$ which correspond under the Springer correspondence to the unipotent class $6631$, that is character sheaves such that $(\Aa^{\Ql}_{E'})_{\H_{uni}} = IC(\overline{(u)_\H}, \mathcal{E'})[\dim(T)-\dim((u)_\H)]$ for $E' \in \irr(W^{{\H}})$ and $\mathcal{E'}$ a local system on $(u)_\H$. Indeed, if $v \in \H$ is unipotent such that $(u)_\H \not \subseteq \overline{(v)_\H}$, then $IC(\overline{(v)_\H}, \mathcal{E'})_{\vert (u)_\H}= 0$ for $\mathcal{E'}$ a local system on $(v)_\H$. On the other hand, if $(u)_\H \subseteq \overline{(v)_\H}- (v)_\H$, then $(u)_\G \subseteq \overline{(v)_\G} - (v)_\G$. By definition of the unipotent support, we must have $s^*((\Aa_E)_{\vert s (v)_\H}) = 0$. Thus the character sheaf of the form $\Aa^{\Ql}_{E'} = IC(\overline{(v)_\H}, \mathcal{E'})[\dim(T)-\dim((u)_\H)]$ can not appear in the decomposition of $s^*((\Aa_E)_{\vert s \H_{uni}})$. 
\\ 

By the Springer correspondence, there is only one character $E' \in \irr(W^{{\H}})$ such that $\Aa^{\Ql}_{E'}[-\dim(\T)]$ is of the form $IC(\overline{(u)_\H}, \mathcal{E'})[\dim (u)_\H]$ for $E' \in \irr(W^{{\H}})$ and $\mathcal{E'}$ a local system on $(u)_\H$. In that case, $\Aa^{\Ql}_{E'} = IC(\overline{(u)_\H}, \Ll_{trivial})[\dim (u)_\H]$ where $\Ll_{trivial}$ is the local system on $(u)_\H$ corresponding to the trivial character of $A_\H(u) = C_\H(u)/C^\circ_\H(u)$. Using CHEVIE, we conclude that $\langle \Ind^{W}_{W_{{\Ll}}} (\phi_{30,15},111) ,	\Ind^{W}_{W^{{\H}}}E'  \rangle_{W} = 0$, whence \[\Aa_{(s u)_\H} = 0.\]

We see that there is an $F$-stable character sheaf $\mathcal{A}$ with unipotent support $\Cc$ such that 
\begin{itemize}
	\item $\mathcal{A}_{\vert ({s} u_\Cc)_{C_\G^\circ({s})}} = 0$, and 
	\item $\mathcal{A}_{\vert(u_\Cc)_\G}[-\dim(\Cc)-\dim(\T)]$ is a local system corresponding to the trivial character of $A_\G(u_\Cc)$. 
\end{itemize}
Therefore, we can apply the same proof as for \cite[Prop.~8.8]{brunatUnitriangularShapeDecomposition2020}, and we get that for all $[b,\phi] \in \Mm(A)$ \[\langle F^G_{[b,\phi]}, D_\G(\chi_\mathcal{A}) \rangle = \begin{cases}
	x_\mathcal{A}  &\text{ if } [b,\phi] =[1,1] \\
	0 &\text{ otherwise,}
\end{cases}\]
for some $x_\mathcal{A} \in \C^\times$.

\subsection{The unipotent conjugacy class $E_7(a_5)$ when $\ell = 2$} 
 We fix the setting in the case where the $F$-stable unipotent class is $\Cc = E_7(a_5)$. We choose $t \in \T^*$ such that $C_{\G^*}(t)$ is of type $E_6 \times A_2$ and $v \in C_{\G^*}(t)$ lies in the unipotent class $D_4(a_1),11$ of $E_7 \times A_1$.

\subsubsection{Admissible covering}
We fix $\mathbf{M}$ the Levi subgroup of $\G$ of type $E_7$. We fix an element $u_\Cc \in \mathbf{M}^F$ such that $(u_\Cc)_\G$ is the unipotent conjugacy class $E_7(a_5)$ and $F$ acts trivially on $A_\mathbf{M}(u_\Cc) \cong S_3$. We write $\Cc_{\mathbf{M}} \defeq (u_\Cc)_{\mathbf{M}}$. Then the unipotent conjugacy class $\Cc_{\mathbf{M}}$ is distinguished in $\mathbf{M}$. We fix a cocharacter $\lambda \in Y_{\mathcal{D}}^{\mathbf{M}}(u)^F$. By the same reasoning as in \cite[$\S$ 10.4]{brunatUnitriangularShapeDecomposition2020}, the group $A = C_{\L_{\mathbf{M}}(\lambda)}(u_\Cc)$ is an admissible covering of $A_{\mathbf{M}}(u_\Cc)$. Then, by the argument at the end of \cite[$\S$ 10.5]{brunatUnitriangularShapeDecomposition2020}, where they apply \cite[Lem.~4.4]{brunatUnitriangularShapeDecomposition2020}, the admissible pair $(A,\lambda)$ is also an admissible covering for $A_\G(u_\Cc)$. 
Observe that $$A \cong A_\G(u_\Cc) \cong \Omega_{C_{\G^*(t)},(v)_{C_{\G^*}(t)}}$$ and if $\ell = 2$, then $ A \cong \Omega^\ell_{\G,\Cc}.$

\begin{rmk} 
	Observe that by \cite[Thm. 1]{mcninchComponentGroupsUnipotent2003} we have, $$A_\G(u_\Cc) \cong \langle \omega_{\alpha_1}(1/2)^{h_1}C_\G^\circ(u_\Cc), \omega_{\alpha_2}(1/3)^{h_2}C_\G^\circ(u_\Cc) \rangle$$ for some $h_1,h_2 \in \G$.  
\end{rmk}

\subsubsection{Character sheaves }
We fix $\Ll \defeq \alpha(s) \in \mathcal{S}(\T)$.
We consider a principal series character sheaf of $\G$ coming from $\mathcal{L}$ with unipotent support $E_7(a_5)$. Thanks to Lusztig's map (\cite[Thm.~10.7]{lusztigUnipotentSupportIrreducible1992}), we choose the one corresponding to the character of $W_\Ll$ denoted by $\phi_{315,16},11$. In order to apply the same argument as in the proof of \cite[Prop.~8.8]{brunatUnitriangularShapeDecomposition2020}, we need to compute the value of the characteristic function of $\Aa \defeq \Aa^\Ll_{\phi_{315,16},11}$ on the conjugacy classes $(au_\Cc)_{C^\circ_\G(a)}$ for each $a \in A$.\\ 


Let's look at the case where $a$ in an involution. Then, there exists $x \in \G$ such that $a^x = \omega_{\alpha_1}(1/2) \in \T$.  We fix ${s} \defeq \omega_{\alpha_1}(1/2) \in \T $ such that $\H \defeq C_\G({s})$ is of type $D_8$. Thanks to Algorithm 5.2 in \cite{brunatUnitriangularShapeDecomposition2020}, since only one unipotent class of $\H$ fuses into $\Cc$, we know that $u \defeq u_\Cc^x$ lies in the unipotent class $7522$ of $\H$ which fuses into $E_7(a_5)$. \\
We want to compute the restriction of the previous character sheaf to the mixed conjugacy class $(s u)_\H$.
We compute $W^{\H}\backslash W /W_{\Ll} = \{1,g\}$ for some $g \in W$. The group $W^{\H} \cap W_{\Ll}$ is a Weyl group. On the other hand, $W^g_\Ll \defeq W^{\H^g} \cap W_{\Ll} $ is not a Weyl group and we have $W^g_\Ll/(W^g_\Ll)^\circ \cong C_2 \cong Z(\H^g)$. Thus, by Lemma \ref{lem_valchiz}, $\chi^s$ is the sign character. \\
Using CHEVIE, we can compute that\footnote{if we did not tensor by the sign character when applying Proposition \ref{prop_formula}, we would have had $(\Aa)_{({s} u)_\H} \neq 0$} \[(\Aa)_{(s u)_{\H}} = 0. \] 

Lastly, we consider the case where $a$ has order $3$. By similar arguments as before, using CHEVIE, we compute that \[(\Aa_{(a u_\Cc)_{C_\G(a)}} = 0. \]

We see that there is an $F$-stable  character sheaf $\mathcal{A}$ with unipotent support $\Cc$ such that 
\begin{itemize}
	\item $\mathcal{A}_{\vert (a u_\Cc)_{C_\G^\circ({s})}} = 0$ if $a \neq 1$ for any $a \in A$.  
	\item $\mathcal{A}_{\vert(u_\Cc)_\G}[-\dim(\Cc)-\dim(\T)]$ is a local system corresponding to the trivial character of $A_\G(u_\Cc)$. 
\end{itemize}
Therefore, we can apply the same proof as for \cite[Prop.~8.8]{brunatUnitriangularShapeDecomposition2020}, and we get that for all $[b,\phi] \in \Mm(A)$ \[\langle F^G_{[b,\phi]}, D_\G(\chi_\mathcal{A}) \rangle = \begin{cases}
	x_\mathcal{A}  &\text{ if } [b,\phi] =[1,1] \\
	0 &\text{ otherwise,}
\end{cases}\]
for some $x_\mathcal{A} \in \C^\times$.\\

\section{Unitriangularity of the $\ell$-decomposition matrix of the unipotent $\ell$-blocks}


We are now ready to prove our main result.
\begin{prop}\label{prop_1}
Let $\G$ be an adjoint simple group of exceptional type defined over $k$, an algebraically closed field of characteristic $p$ with Frobenius endomorphism $F$. Assume that $p$ is good for $\G$. If  $\ell$ is bad for $\G$, then the decomposition matrix of the unipotent $\ell$-blocks of $G$ is unitriangular. 
\end{prop}

\begin{proof} We fix a total ordering of the $\ell$-special unipotent conjugacy classes of $\G$, $\Cc_1,\dots, \Cc_r$ such that $n <m$ if $\dim(\Cc_n) < \dim(\Cc_m)$.
	 	
Let $\Cc_n$ be a unipotent $\ell$-special conjugacy class and $\alpha_{n} \defeq |\Mm^\ell(\Omega^\ell_{\G,\Cc_n})|$.
Thanks to our previous discussion, we can find projective $\k G$-modules $P^n_1,\dots, P^n_{\alpha_{n}}$ with characters $\pi^n_j$ associated to their lift to $\K G$-modules and irreducible characters of $G$ in the unipotent $\ell$-blocks with unipotent support $\Cc_n$, $\rho^n_1,\dots,\rho^n_{\alpha_{n}}$, such that for all $1 \leq i,j \leq \alpha_n$ 
\[\langle (\rho^n_i)^*, \pi^n_j \rangle = \begin{cases}
0 \quad i<j,\\
1 \quad i =j.
\end{cases}\]
In particular, for a fixed $n$ the $P^n_i$ are all distinct. 

Let $\Cc_m \neq \Cc_n$ be another unipotent $\ell$-special conjugacy class of $\G$ and $\rho'$ be an irreducible character of $G$ with unipotent support $\Cc_m$. Suppose that there is  $1\leq j\leq \alpha_{n}$, such that $\langle (\rho')^*, \pi^n_j \rangle  \neq 0$.\\
We observe that if $\langle (\rho')^*, \pi^n_j \rangle  \neq 0$, then there exists $v \in \Cc_n^F$ and a generalised Gelfand--Graev character $\gamma_v$, such that $\langle (\rho'), \gamma_v \rangle  \neq 0$. If $P^n_j$ is itself a GGGC, then it is obvious. Otherwise it is a consequence of Lemma \ref{Kaw_char} and Equation (\ref{KawGGGr}). 
In any case, since $(\rho')^*$ has wave front set $\Cc_m$, we conclude by the unicity of the wave front set \cite[Thm 15.2]{taylorGeneralizedGelfandGraevRepresentations2016}) that  $(v)_\G = \Cc_n \subseteq \overline{\Cc_m}$, whence $\dim(\Cc_n) < \dim(\Cc_m)$ and thus $n<m$.

Now for each $1 \leq n \leq r$ and $1 \leq i \leq \alpha_n$, we set $\mu^n_i \defeq (\rho^n_i)^*$. The irreducible character $\mu^n_i$ lies in the unipotent $\ell$-blocks. Moreover, for $1 \leq m \leq r$ and $1 \leq j \leq \alpha_m$, we have \[\langle \mu^n_i, \pi^m_j \rangle = \begin{cases}
0 \text{ if } n <m \text{ or } (n =m \text{ and }i<j),\\
1 \text{ if } n=m \text{ and } i =j.
\end{cases}\]
Therefore, summing over all the $\ell$-special unipotent conjugacy classes, we obtain the exact number of indecomposable projective $\k G$-modules in the unipotent $\ell$-blocks.
We conclude thanks to Proposition \ref{prop_projsuff}.
\end{proof}



\pagebreak[1]
\appendix

\section{Tables for the exceptional groups}\label{Tables}

\begin{table}[h!]
	\centering
	$\begin{array}{|l|l|l|l|}
		\hline
		\mathcal{C} & A_G(u) & \Omega^2_u & \Omega^3_u \\ \hline
	A_5 & 1  & 1 &    \\ \hline
	A_3+A_1 & 1   & 1 &      \\ \hline
	3A_1 & 1   & 1 &    \\ \hline
	2A_2+A_1 & 1  &  & 1    \\ \hline
	\end{array}$
	\caption{The $\ell$-special but not special classes of  $E_6$}
	\label{tableE6}
\end{table}
\FloatBarrier

\begin{table}[h!]
	\centering
	$\begin{array}{|l|l|l|l|}
		\hline
			\mathcal{C} & A_G(u)  & \Omega^2_u & \Omega^3_u  \\ \hline
	D_6 & 1 &   1 &      \\ \hline
	D_6(a_2) & 1   & 1 &     \\ \hline
	A_5' & 1  & 1 &     \\ \hline
	D_4+A_1 & 1  & 1 &      \\ \hline
	A_3+2A_1 & 1  & 1 &      \\ \hline
	(A_3+A_1)' & 1  & 1 &     \\ \hline
	4A_1 & 1  & 1 &      \\ \hline
	3A_1' & 1  & 1 &      \\ \hline
	A_5+A_1 & 1  &  & 1    \\ \hline
	2A_2+A_1 & 1  &  & 1    \\ \hline
	\end{array}$
	\caption{The $\ell$-special but not special classes of $E_7$}
	\label{tableE7}
\end{table}
\FloatBarrier

\begin{table}[h!]
	\centering
	$\begin{array}{|l|l|l|l|l|}
	\hline
	\mathcal{C} & A_G(u) & \Omega^2_u & \Omega^3_u & \Omega^5_u  \\ \hline
	   E_7 & 1   & 1 &  &     \\ \hline
	D_7 & 1 &   1 &  &      \\ \hline
	E_7(a_2) & 1  & 1 &  &      \\ \hline
	D_6 & 1  & 1 &  &    \\ \hline
	A_7 & 1   & 1 &  &      \\ \hline
	D_5+A_1 & 1   & 1 &  &     \\ \hline
	E_7(a_5) & S_3   & S_3 &  &      \\ \hline
	D_6(a_2) & S_2   & S_2 &  &      \\ \hline
	D_5(a_1)+A_2 & 1   & 1 &  &      \\ \hline
	A_5 & 1   & 1 &  &      \\ \hline
	D_4+A_1 & 1   & 1 &  &      \\ \hline
	2A_3 & 1   & 1 &  &     \\ \hline
	A_3+A_2+A_1 & 1   & 1 &  &     \\ \hline
	A_3+2A_1 & 1   & 1 &  &      \\ \hline
	A_3+A_1 & 1  & 1 &  &    \\ \hline
	A_2+3A_1 & 1   & 1 &  &     \\ \hline
	4A_1 & 1 &   1 &  &      \\ \hline
	3A_1 & 1 &   1 &  &    \\ \hline
	E_6+A_1 & 1 &    & 1 &      \\ \hline
	E_6(a_3)+A_1 & S_2 &    & S_2 &      \\ \hline
	2A_2+2A_1 & 1 &    & 1 &    \\ \hline
	2A_2+A_1 & 1 &    & 1 &     \\ \hline
	A_4+A_3 & 1 &   &  & 1     \\ \hline
	\end{array}$
	\caption{The $\ell$-special but not special classes of $E_8$}
	\label{tablee83}
\end{table}
\FloatBarrier

\FloatBarrier

\begin{table}[h!]
	\centering
	$\begin{array}{|l|l|l|l|}
		\hline
		\mathcal{C} & A_G(u) &  \Omega^2_u & \Omega^3_u  \\ \hline
		\tilde{A}_1 & 1   & 1 &     \\ \hline
		A_1 & 1 &    & 1     \\ \hline
	\end{array}$
	\caption{The $\ell$-special but not special classes of $G_2$}
	\label{tableG2}
\end{table}
\FloatBarrier

\begin{table}[h!]
	\centering
	$\begin{array}{|l|l|l|l|}
		\hline
		\mathcal{C} & A_G(u) &  \Omega^2_u & \Omega^3_u \\ \hline
		A_1 & 1  & 1 &   \\ \hline
		A_2+\tilde{A}_1 & 1  & 1 &      \\ \hline
		B_2 & S_2   & S_2 &    \\ \hline
		C_3(a_1) & S_2  & S_2 &      \\ \hline
		\tilde{A}_2+A_1 & 1   &  & 1    \\ \hline
	\end{array}$
	\caption{The $\ell$-special but not special classes of $F_4$}
	\label{tableF4}
\end{table}
\FloatBarrier

	\printbibliography

\end{document}